\documentclass[11pt,a4paper]{article}

\usepackage{amsthm,natbib}

\usepackage{latexsym,amssymb,lastpage}
\usepackage{graphicx,amsfonts}
\usepackage{times,bm,amsmath}

\renewenvironment{proof}[1][Proof]{\noindent\textit{#1. } }{\hfill$\square$}

 \newtheoremstyle{theorem}{6pt}{6pt}{\rm}{}{\sffamily}{ }{ }{}
 \theoremstyle{theorem}
\newtheorem{theorem}{\sc Theorem}[section]

  \newtheoremstyle{thm}{6pt}{6pt}{\rm}{}{\sffamily}{ }{ }{}
 \theoremstyle{thm}

 \newtheoremstyle{lemma}{6pt}{6pt}{\rm}{}{\sffamily}{ }{ }{}
 \theoremstyle{lemma}
 \newtheorem{lemma}{\sc Lemma}[section]
 \newtheoremstyle{lem}{6pt}{6pt}{\rm}{}{\sffamily}{ }{ }{}
 \theoremstyle{lem}

\newtheoremstyle{case}{6pt}{6pt}{\rm}{}{}{. }{ }{}
 \theoremstyle{case}

 \newtheoremstyle{statement}{6pt}{6pt}{\rm}{}{\sffamily}{ }{ }{}
\theoremstyle{statement}

 \newtheoremstyle{corollary}{6pt}{6pt}{\rm}{}{\sffamily}{ }{ }{}
 \theoremstyle{corollary}

  \newtheoremstyle{defi}{6pt}{6pt}{\rm}{}{\sffamily}{ }{ }{}
 \theoremstyle{defi}

  \newtheoremstyle{cor}{6pt}{6pt}{\rm}{}{\sffamily}{ }{ }{}
 \theoremstyle{cor}

\newtheorem{definition}[theorem]{Definition}

\newtheoremstyle{example}{6pt}{6pt}{\rm}{}{\sffamily}{ }{ }{}
\theoremstyle{example}

\newtheorem{proposition}[theorem]{\sc Proposition}

\newtheoremstyle{remark}{6pt}{6pt}{\rm}{}{\sffamily}{ }{ }{}
\theoremstyle{remark}
\newtheorem{remark}{\sc Remark}[section]

\newtheoremstyle{approximation}{6pt}{6pt}{\rm}{}{\sffamily}{ }{ }{}
\theoremstyle{approximation}

\newtheoremstyle{scheme}{6pt}{6pt}{\rm}{}{\sffamily}{ }{ }{}
\theoremstyle{scheme}

\newtheoremstyle{Algorithm}{6pt}{6pt}{\rm}{}{\sffamily}{ }{ }{}
\theoremstyle{Algorithm}

 \newtheoremstyle{Remark}{6pt}{6pt}{\rm}{}{\sffamily}{ }{ }{}
 \theoremstyle{Remark}

\newtheoremstyle{Lemma}{6pt}{6pt}{\rm}{}{\sffamily}{ }{ }{}
\theoremstyle{Lemma}

\newtheoremstyle{Assumption}{6pt}{6pt}{\rm}{}{\sffamily}{ }{ }{}
\theoremstyle{Assumption}

\newtheoremstyle{Proposition}{6pt}{6pt}{\rm}{}{\sffamily}{ }{ }{}
\theoremstyle{Proposition}

\newtheoremstyle{prop}{6pt}{6pt}{\rm}{}{\sffamily}{ }{ }{}
\theoremstyle{prop}

\newtheoremstyle{rem}{6pt}{6pt}{\rm}{}{\sffamily}{ }{ }{}
 \theoremstyle{rem}

\newtheoremstyle{hypo}{6pt}{6pt}{\rm}{}{\sffamily}{ }{ }{}
 \theoremstyle{hypo}

  \newtheoremstyle{Step}{6pt}{6pt}{\rm}{}{}{ }{ }{}
 \theoremstyle{Step}

 \newtheoremstyle{lema}{6pt}{6pt}{\rm}{}{\sffamily}{ }{ }{}
 \theoremstyle{lema}

\renewcommand{\theequation}{\thesection.\arabic{equation}}
\numberwithin{equation}{section}

\usepackage{fullpage}
\usepackage{color}
\usepackage{mathrsfs}
\usepackage{dsfont}
\usepackage{tikz}

\renewcommand{\d}{\,\mathrm{d}}  

\renewcommand{\d}[1]{\mathrm{#1}}
\newcommand{\mc}[1]{\mathcal{#1}}
\renewcommand{\div}{\text{div}}
\newcommand{\Ga}{\Gamma}
\newcommand{\Om}{\Omega}
\newcommand{\lay}{\Omega\setminus\overline{\Omega_{\delta}}}
\renewcommand{\O}{\mathcal{O}}
\newcommand{\ve}{\varepsilon}
\newcommand{\Omext}{\Om_{\text{ext}}}
\newcommand{\vp}{\varphi}
\newcommand{\om}{\omega}

\newcommand{\ld}{\lambda}

\newcommand{\vapp}{v_{{\rm app}}}
\newcommand{\wapp}{w_{{\rm app}}}

\newcommand{\ga}{\gamma}
\newcommand{\bR}{\mathbb{R}}

\newcommand{\wt}[1]{\widetilde{#1}}
\renewcommand{\b}[1]{\overline{#1}}
\newcommand{\D}[2]{\frac{\partial #1}{\partial #2}}
\newcommand{\nnn}{\text{\boldmath{$\nu$}}}
\renewcommand{\tt}{\text{\boldmath{$\tau$}}}
\renewcommand{\ni}{\text{\boldmath{$\nu$}}}

\begin{document}

\title{Asymptotic analysis of the transmission eigenvalue problem for a Dirichlet obstacle coated by a thin layer of non-absorbing media}
\author{{\sc Fioralba Cakoni$^{1}$,  Nicolas Chaulet $^2$ and Houssem Haddar $^3$}\\[4pt]
$^1$Department of Mathematical Sciences, University of Delaware, \\[2pt]
Newark, DE 19716, USA.\\[4pt]
$^2$Department of Mathematics University College London,  \\[2pt]
Gower street, London, WC1E 6BT, UK.\\[4pt]
$^3$ INRIA Saclay Ile de France / CMAP Ecole Polytechnique, \\[2pt] Route de Saclay, 91128
Palaiseau Cedex, France.
}
\maketitle

\begin{abstract}
{We consider the transmission eigenvalue problem for an impenetrable obstacle with Dirichlet boundary condition surrounded  by a thin layer of non-absorbing 
inhomogeneous material.  We derive a rigorous asymptotic expansion for the
first transmission eigenvalue  with respect to the thickness of the thin layer. Our 
convergence analysis is based on a Max-Min principle and an iterative approach which involves estimates
on the corresponding eigenfunctions. We provide explicit expressions for 
the terms  in the asymptotic expansion up to order three. }
{Transmission eigenvalues, thin layers, asymptotic methods, inverse scattering}
\end{abstract}

\section{Introduction}
Transmission eigenvalues appear in the study of scattering by inhomogeneous media and are closely related to non-scattering frequencies \cite{CakHad2}, 
\cite{sp}. Such eigenvalues  provide information about material properties of the scattering media  \cite{cakginhad} and can be determined from scattering data 
\cite{CaCoHa10}, \cite{Leh}. Hence they can play an important role in a variety of inverse problems in target identification and non-destructive testing 
\cite{haddex}. The transmission eigenvalue problem is a non-selfadjoint and nonlinear problem that is not covered by the standard theory of eigenvalue problems 
for elliptic operators. In the past few years transmission eigenvalues have become an important area of research in inverse scattering theory. Since the first proof 
of existence of transmission eigenvalues in \cite{cakginhad} and  \cite{paisyl}, the interest in the transmission eigenvalue problem has increased, resulting in a 
number of important advancements. For an update survey on the topic we refer the reader to  \cite{CakHad2}. 

In this paper we consider the transmission eigenvalue problem corresponding to the scattering by an impenetrable obstacle with Dirichlet boundary condition 
coated by a thin layer of non-absorbing inhomogeneous material. The existence and discreteness of transmission 
eigenvalue problem  is investigated in   \cite{CaCoHa10} (see also 
\cite{LakVain13}).  In the two-dimensional case this  problem models the  scattering of TE-polarized 
electromagnetic waves (written in terms of the electric field) by an infinitely long  cylindrical prefect conductor coated by a thin layer of non-magnetic dielectric 
material. In the three dimensional case it models the scattering of acoustic waves by a sound-soft object surrounded by acoustically non-absorbing material. It is well known (see e.g. \cite{BenLem96}) that the first order approximation to the  scattering problem for a coated perfect conductor is an exterior 
boundary value problem with impedance type boundary condition where the impedance function depends inverse proportionally to the thickness of the layer, 
here denoted by $\delta$. The corresponding "non-scattering" frequencies for this approximate model become the eigenvalues of a non-coercive Robin 
eigenvalue problem, which is studied  by the authors of this paper in \cite{CCH13}. 

The main concern of this study is to develop a rigorous asymptotic expansion for transmission eigenvalues  as $\delta\to 0$. Our asymptotic analysis is based on 
an iterative and constructive approach. We restrict ourselves here to the first
transmission eigenvalue.  As expected this transmission eigenvalue is  close to
the first  Dirichlet eigenvalue up to  order 
$\delta$, result that is proven directly in this paper by using the Max-Min principle. Then, the main idea of our approach is, roughly speaking, having proven  
convergence  of order $k$ for  the asymptotic expansion of the transmission eigenvalue, we next prove  estimates of order $k$ for the corresponding eigenfunctions 
by using standard approximation results for the eigenfunctions of the negative Laplacian with Dirichlet boundary conditions. Then, we deduce convergence at order 
$k+1$ for the eigenvalues by using the Max-Min principle. Although our analysis can in principle be carried through for any 
order, for sake of simplicity we provide here explicit expressions only for the terms up to order three in the asymptotic expansion of the first transmission 
eigenvalue. The explicit construction of the asymptotic expansion is simplified
by the fact that the first  eigenvalue of the Dirichlet problem is simple. The
extension of our analysis to higher order transmission eigenvalues is
challenging, first because explicit construction of the asymptotic is
complicated and second because one looses the characterization of the
transmission eigenvalues in terms of a Max-Min principle. 

From practical point of view, the second order expansion provides in fact  a formula for the thickness of  the layer in terms of the first (measurable) transmission 
eigenvalue. Unfortunately, the refractive index of the layer does not appear in the first three terms of the asymptotic expansion. Of course, the refractive index will 
show in higher order terms but then the obtained reconstruction formula would be highly unstable with respect to noise in the transmission eigenvalue.  
A better model to capture both the 
thickness and the refractive index in the first order term in the context of electromagnetic scattering is to write the problem in terms of the magnetic field, which 
would lead to Neumann boundary condition on the boundary of the inclusion. Unfortunately the transmission eigenvalue problem for inhomogeneous media 
containing an inclusion with Neumann boundary condition  is still
open. Moreover, no Max-Min principle is available in this case which is the corner stone of our 
approach. 

The structure of the paper is as follows. In the next section we formulate the problem and recall some relevant results on the transmission eigenvalue problem for 
an inhomogeneous media containing an inclusion with Dirichlet boundary condition. In Section 3 we derive the formal asymptotic expansion for  transmission 
eigenvalues and provide explicit formulas for the terms up to order three.  Section 4 is dedicated to the rigorous convergence 
proof of the asymptotic expansion derived in the previous section  for the first transmission eigenvalue.
 For our analysis we need various technical results that to our knowledge are not available in 
the literature, in particular elliptic a priori estimates and trace lemma with explicit dependance on $\delta$.
 To keep the reader focused in the main goal of the paper, 
we prove all the auxiliary results needed for our analysis in  Appendix.  

\section{Formulation of the problem}
\label{sec:description}
We consider an impenetrable object  coated with a thin layer of non-absorbing penetrable material  with refractive index $n$ which occupies the region $\Om
\subset {\mathbb R}^d$, $d=2,3$ where $\Om$ is bounded and simply connected  with smooth enough (to become precise later) boundary $\Ga$. We denote by 
$$\Om_\delta = \{ x \in \Om \text{ such that } d(x,\Ga) >\delta \}$$ and by  $$\Ga_\delta = \{x \in \Om \text{ such that } d(x,\Ga) 
= \delta\}$$
its boundary.  The simply connected domain $\Om_\delta$ (see Figure \ref{fig:mince})  represents here  the impenetrable object and $\lay$ represents the thin layer.  
\begin{figure}[hh]
\begin{center}
\begin{tikzpicture}
 \filldraw [fill=gray!40,even odd rule,smooth cycle,tension=.8] plot coordinates {(-2,-1) (1.5,-1.5) (1.5,0)(-2.,2)};
   \path [draw,fill=white,smooth cycle,tension=.8,] plot coordinates {(-1.8,-0.8) (1.3,-1.3) (1.3,-0.2)(-1.85,1.8)};
 \draw[thick,->] (0.3,1) -- (-0.05,0.5); 
 \node at (-0.35,0.5) {$\nnn$};
  \draw[thick,<->] (1.3,-0.2) -- (1.5,0); 
   \node at (1.75,0.25) {$\delta$};
 \draw[thick] (-2.25,-0.25) -- (-2.75,-0.5); 
 \node at (-3.8,-0.75) {$\lay$};
  \node at (-2.85,1) {$\Ga$}; 
    \node at (-1.9,0.6) {$\Ga_\delta$};
    \node at (-0.5,0.) {$\Om_\delta$};
        \node at (1.5,1) {$\Omext$};
\end{tikzpicture}
\end{center}
\caption{The scattering layered object}
 \label{fig:mince}
\end{figure}
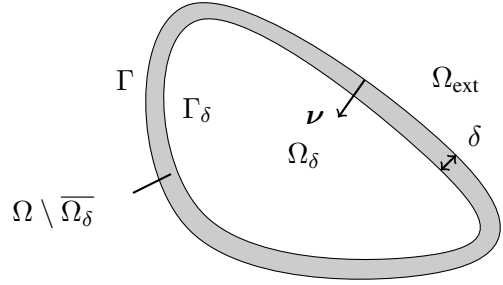
The scattering of an incident wave $u^i$, which here for simplicity is assumed to be an entire solution of the Helmholtz equation (one could also consider the 
incident field to  be a point source located outside $\Omega$), by such a  structure gives rise to a scattered field $u^s=u-u^i$, with  $u$ being the total field, that 
satisfies
\begin{equation}
\label{eq:coucheminceITEP}
\begin{cases}
\Delta u+ k^2 n u=0 \text{ in }  \lay, \\
\Delta u  + k^2 u=0 \text{ in } \Omext := \bR^d\setminus \b{\Om}, \\
\displaystyle  \left[ \D{ u}{\nnn} \right] = 0 \,, \quad [u ]=0 \text{ on }  \Ga,\\
u =0 \text{ on } \Ga_\delta,\\
\displaystyle \lim \limits_{R \to \infty} \int_{|x|=R}\left|\partial_{r} u^s - iku^s \right|^2ds =0.
\end{cases}
\end{equation}
where $k$ is the wave number, $n\in L^\infty(\lay)$ is the index of refraction of the layer such that $n\geq n_0>0$, $\nnn$ is unitary normal to $\Ga$ directed 
inward to $\Om$ and  $[v]=v^+-v^-$ denotes the jump of $v$ across $\Ga$ where $v^+$ is the exterior trace of $v$ and $v^-$ is the interior trace of $v$ on $\Ga$. 
The corresponding transmission eigenvalue problem  is to find the  values of
$k^2_\delta$ such that there exists a non trivial solution
$(w_\delta,v_\delta) \in L^2(\lay)\times L^2(\Om)$  to the 
following homogeneous coupled problem
\begin{equation}
\label{eq:ITEP}
\begin{cases}
\Delta w_\delta+ k_\delta^2 n w_\delta=0 \text{ in }  \lay, \\
\Delta v_\delta  + k_\delta^2 v_\delta=0 \text{ in } \Om, \\
\displaystyle  \D{v_\delta}{\nnn} = \D{w_\delta}{\nnn} \,, \quad v_\delta=w_\delta \text{ on }  \Ga,\\
w_\delta =0 \text{ on } \Ga_\delta.
\end{cases}
\end{equation}

\begin{definition}
\label{def:ITE}
The values $k^2_\delta>0$ for which \eqref{eq:ITEP} has a non trivial solution $(w_\delta,v_\delta)\in L^2(\lay)\times L^2(\Om)$ are called transmission eigenvalues, and the nonzero 
solutions  $w_\delta$ and $v_\delta$ the associated eigenfunctions.
\end{definition}
It is shown in  \cite{CaCoHa10}  that the real transmission eigenvalues (the wavenumber $k$ is related to the interrogating frequency) can be determine 
from measured far field (or near field) scattering data. Note that the transmission eigenvalue problem is non-selfadjoint and complex eigenvalues may occur but, 
from practical point of view as discussed in Introduction and  the fact that only real transmission eigenvalues are proven to exist, here we are interested only 
on real transmission eigenvalues (see e.g. \cite{CakHad2}). Our main goal in this paper is to derive rigorous  asymptotic expansions for transmission eigenvalues 
in terms of the thickness of the layer $\delta$  as $\delta\to 0$.
 
The transmission eigenvalue problem for an inhomogeneity containing an  impenetrable inclusion with Dirichlet boundary condition is investigated in 
\cite{CaCoHa12}  and \cite{LakVain13}  (our problem \eqref{eq:ITEP} is exactly of that form)  where the discreteness and existence of real transmission eigenvalues is shown  under 
appropriate assumptions on the refractive index $n$. For the sake of reader's convenience and later use  we summarize the main results from \cite{CaCoHa12}.

The first step in the analysis of \eqref{eq:ITEP} consists in reformulating   it as  an eigenvalue problem for a forth order equation. To this end, introducing
\begin{equation}
u_\delta = 
\begin{cases}
w_\delta - v_\delta \text{ in } \lay\\
-v_\delta \text{ in } \Om_\delta
\end{cases}
\label{eq:defu}
\end{equation}
we obtain that this $u_\delta$ satisfies
\begin{equation}
\label{eq:bilpludelta}
(\Delta + k_\delta^2)\frac{1}{1-n}(\Delta+ k_\delta^2n)u_\delta =0 \text{ in } \lay.
\end{equation}
Equation (\ref{eq:bilpludelta}) together with the fact that   $u_\delta$ must be in $H^{1}_{0}(\Om)$  and satisfy the Helmholtz equation in $\Om_
\delta$ suggest that  to arrive at a variational formulation equivalent to the eigenvalue problem \eqref{eq:ITEP} we need to introduce the space 
\[
 W_{\delta}:=\left\{  u\in H^{1}_{0}(\Om)\cap H^{1}_{\Delta}(\lay) \;\text{such that}\; \D{u}{\nnn}=0 \;\text{on}\;\Ga \right\}
\]
equipped with  the norm
\[
\|u\|^{2}_{W_{\delta}} := \|u\|^{2}_{H^{1}(\Om)} + \|\Delta u\|^{2}_{L^{2}(\lay)}.
\]
Then it is shown in  \cite{CaCoHa12}  that  $k^2_\delta>0$ is a  transmission eigenvalue (according to Definition \ref{def:ITE}) with associated  eigenfunctions $
(w_\delta,v_\delta)$ if and only if  $u_{\delta}$ defined by \eqref{eq:defu} solves 
\begin{equation}\label{misp}
A_{k_{\delta}} u_{\delta} - k_\delta^4 Bu_{\delta}=0
\end{equation}
  where  the bounded linear self-adjoint operators $A_{k} \,:\,W_{\delta} \rightarrow W_{\delta}$ and $B  \,:\,W_{\delta} \rightarrow W_{\delta}$ are given by
\[
(A_{k} u,v)_{W_{\delta}} := \int_{\lay}\frac{1}{1-n} (\Delta u+k^2 u)(\Delta \b{v} +k^2\b{v}) \,dx + k^{4}\int_{\Om} u\b{v} \, dx + k^2 \int_{\Om} \nabla u\cdot\nabla\b{v}
\,dx,
\]
\[
(B u,v)_{W_{\delta}} := 2\int_{\Om} u\b{v}\,dx.
\]
In the following we denote  $n_*=\inf_{\lay}n(x)$ and $n^*=\sup_{\lay}n(x)$. The operators $A_{k}$ and $B$ satisfy the following properties.
\begin{proposition}
\label{prop:Acoercive}
Assume that $0<n_*<n(x)<n^*<1$. Then  $B$ is a compact operator and there exists a constant $C>0$ such that for all $\delta>0$
\[
|(A_{k} u,u)_{W_{\delta}}| \geq C \|u\|^2_{W_{\delta}}.
\]
\end{proposition}
\begin{proof}
The proof can be found in \cite[Theorem 2.1]{CaCoHa12}. The fact that the coercivity constant $C$ is independent of $\delta$ is clear in this proof.
\end{proof}

We remark that if $k_\delta$ and $u_\delta\neq 0$ satisfy (\ref{misp}),  then $(w_\delta,v_\delta)$ are  obtained from $u_\delta$ by
\begin{align}
w_\delta &= \frac{1}{k^2_\delta(1-n)}(\Delta u_\delta + k^2_\delta u_\delta ) \text{ in } \lay,\label{eq:udeltatowdelta}\\
 v_\delta &= 
 \begin{cases}
-u_\delta \text{ in } \Om_\delta,\\
\displaystyle \frac{1}{k_\delta^2(1-n)}(\Delta u_\delta + k_\delta^2n u_\delta )\text{ in } \lay.
 \end{cases} \label{eq:udeltatovdelta}
\end{align}
The following result  proven in \cite{CaCoHa12} is the starting point of our discussion. 
\begin{theorem}
\label{th:existence}
Assume that $0<n_*<n(x)<n^*<1$. There exist an infinite discrete set of transmission eigenvalues and $+\infty$ is the only accumulation point.
\end{theorem}
At this point we choose to normalize the $w_\delta$ and $v_\delta$ so that 
\[
\|u_\delta\|_{L^2(\Om)}=1.
\] 
The following regularity result for the eigenfunctions $(w_\delta,v_\delta)$ holds true.
 \begin{lemma}
\label{LeEllReg}
Assume that  $\Ga$ is a $C^{k+2}$-boundary and  $n\in C^{k+2}(\overline{\Omega}\setminus \Omega_\delta)$  with $k\geq 2$. Then $w_{\delta} \in H^{k}(\lay)$ 
and $v_{\delta} \in H^{k}(\Om)$.
\end{lemma}
\begin{proof}
First since $\Delta v_{\delta} =-k^2_\delta v_{\delta}$, using interior elliptic regularity for the Laplacian, we have  that $v_{\delta} \in C^{\infty}(\om)$ for all open set 
$\om \subset\subset\Om$. Hence its trace and its normal derivative trace  on $\Ga_{\delta}$ are in $H^{k+2-1/2}(\Ga_{\delta})$ and $H^{k+2-3/2}(\Ga_{\delta})$ 
respectively. Using the same argument but this time for the Laplace operator with homogeneous Dirichlet boundary condition on $\Ga_\delta$ we can conclude 
that the trace of the normal derivative of $w_{\delta}$ on $\Ga_{\delta}$ is also in $H^{k+2-3/2}(\Ga_{\delta})$. Hence we can easily obtain  that on $\Ga_{\delta}$ 
we have
\[
\begin{cases}
\Delta u_{\delta} \in H^{k+2-1/2}(\Ga_{\delta}),\\
\displaystyle{\frac{\partial}{\partial{{\nnn_{\delta}}}}}\Delta u_{\delta}\in H^{k+2-3/2}(\Ga_{\delta})
\end{cases}
\] 
 where $\nnn_\delta$ is the unit normal to $\Ga_\delta$ directed inward $\Om_\delta$. Since $u_{\delta} $ satisfies \eqref{eq:bilpludelta} in $\lay$ with 
homogeneous boundary conditions on $\Ga$, regularity results for the bilaplacian  implies that $u_{\delta}$ is in $H^{k+2}(\lay)$ (see e.g. \cite{Ag}). We finally 
obtain the result by using \eqref{eq:udeltatowdelta} and \eqref{eq:udeltatovdelta}.
 \end{proof}

From now on we assume that the refractive index  satisfies 
\[
0<n_*<n(x)<n^*<1.
\]
This assumption ensure existence of the interior transmission eigenvalues (Theorem \ref{th:existence}) but is more restrictive than the one proposed in 
\cite{LakVain13}  that allows $n$ to be greater than $1$ provided the thickness of the layer is sufficiently small. Nevertheless, when $n>1$ the operator $A_{k_
\delta}$ is sign indefinite and we loose the Max-Min principle which is the main ingredient of our approach.

\section{Formal Asymptotic Expansion}
\subsection{Preliminary material}\label{prelim}
For the sake of simplicity, here we perform the asymptotic expansion in the two dimensional case. The extension to three dimensional case is purely   a  technical 
issue and it is possible to obtain similar asymptotic expansions by using the same approach. Having limited ourselves to the two dimensional case and assuming 
that the boundary is $C^{k+2}$-smooth for $k\geq 2$,  we can parametrize  $\Gamma$ as
\[
\Gamma = \{ x_{\Ga}(s) \;,\; s \in [0;s_{0}] \}
\]  
where the periodic function $x_{\Gamma} \,:\, [0;s_{0}] \rightarrow \bR^{2}$ is in $C^{k+2}([0;s_{0}])$ for some  $s_{0}>0$.
Moreover, we can choose this parameterization such that  the tangent vector $\tt(s) := \frac{d x_{\Ga}}{ds}(s)$  to the surface $\Ga$ at the  arbitrary point $x_{\Ga}
(s)$ is a unit vector. Then denoting by $\ni(s)$ the inward unit  normal vector to $\Gamma$  at the  point $x_{\Ga}(s)$ and we can define the curvature $\kappa(s)$ 
by
\[
\frac{d \tt}{ds}(s) = -\kappa(s) \ni(s).
\]
Based on this parameterization of the curve $\Ga$, we obtain the following parameterization of the surface $\Ga_\delta$ 
\begin{equation}
\label{EqThinLayer}
\Gamma_{\delta} =\{x_{\Ga}(s)+\delta(s)\ni(s) \;,\; s \in [0;s_{0}]\}
\end{equation}
where $\delta \in C^\infty([0;s_{0}])$  is a periodic  function  of sufficiently small values.
Let us define  by
$$\eta_{0}:= \inf_{s\in[0,s_{0}]}\frac{1}{|\kappa(s)|},$$ and $\Om_{0}:=\{x \in \bR^{2}\, ,\; \text{dist}(x,\Ga)\leq \eta_{0}\}$. Then the map 
\begin{align*}
\vp \,:\,  [0,s_{0}]\times[-\eta_{0},\eta_{0}]&\longrightarrow \Om_{0}\\
(s,\eta) &\longmapsto  x_{\Ga}(s) +\eta \ni(s)
\end{align*}
is a $C^{k+2}$-diffeomorphism,  in other words, for every point $x \in \Om_{0}$ there exists a unique $(s,\eta) \in [0,s_{0}]\times[-\eta_{0},\eta_{0}]$ such that
\[
x = x_{\Ga}(s) +\eta \ni(s).
\] 
Next, for any function $u$ defined on $\Om_{0}$ we  can define $\wt{u}$ in  $[0,s_{0}]\times[-\eta_{0},\eta_{0}]$ by 
\begin{equation}\label{tilde}
\wt{u}(s,\eta):=u\circ \vp(s,\eta)
\end{equation}
and the  gradient of $u$ in the local coordinates $(s,\eta)$ writes as
\[
\nabla u = \cfrac{1}{(1+\eta \kappa)}\cfrac{\partial}{\partial s} \wt{u}\tt + \cfrac{\partial}{\partial \eta} \wt{u}\ni.
\]
Furthermore, using integration by parts  we have that the divergence of  a vector field $\vec{u} = u_{\tau} \tt + u_{n} \ni$  writes as
\[
\div \vec{u} = \cfrac{1}{(1+\eta \kappa)}\cfrac{\partial}{\partial s} \wt{u}_{\tau} + \cfrac{1}{(1+\eta \kappa)}\cfrac{\partial}{\partial\eta}(1+\eta \kappa)\wt{u}_{n}.
\]
We finally denote by $J_{s,\eta} := |\det(\nabla \vp (s,\eta))| = 1+\eta \kappa(s)$ the Jacobian of the change of variables.

\subsection{Formal derivation of the asymptotic expansion}
\label{sec:devas}
Let us now turn our attention to the transmission eigenvalue problem (\ref{eq:ITEP}). To be able to carry on our computations, we assume that the function $\delta
$ used in (\ref{EqThinLayer}) to define the interior boundary is of the form $\delta(s) = \delta_{0} g(s)$ for some constant $\delta_{0}>0$ and some strictly positive 
$C^\infty$-function $g$ independent of $\delta_{0}$ such that $|\delta_{0} g(s)| < \eta_{0}$. To simplify the notations and since there is no ambiguity, we make no 
distinction between $g$ as a function of local and  global variables. Then, we postulate the following ansatz for the interior transmission eigenvalues and the 
associated eigenfunctions:
\[
k^2_{\delta} = \sum_{j=0}^{\infty} \delta_{0}^{j} \ld_{j}
\]
\begin{equation}
\label{EqExpensionw}
w_{\delta}(x) =\hat{w}_{\delta}(s,\xi)=\sum_{j=0}^{\infty} \delta_{0}^{j}\hat{w}_{j}\left(s,\xi\right)
\end{equation}
\[
v_{\delta}(x) =\sum_{j=0}^{\infty} \delta_{0}^{j}v_{j}(x) 
\]
for $\xi = \eta/\delta_{0}$. We remark that the functions $\hat w_{j}$  are defined on $\mathcal{G}:=\{(s,\xi) \in [0,s_{0}]\times[0,\max(g)]\,,\; \xi \leq g(s)\}$ which is 
independent of $\delta_0$ and we define $w_k(x) = \hat{w}_k(s,\eta/\delta_0)$. Using (\ref{eq:ITEP}) and the expressions for the gradient and divergence 
operators in the local coordinates, we obtain that $(\hat{w}_{\delta},v_\delta)$ satisfies
\begin{equation}
\label{EqLocal}
 \cfrac{1}{(1+\xi \delta_{0} \kappa)}\cfrac{\partial}{\partial s} \cfrac{1}{(1+\xi \delta_{0} \kappa)}\cfrac{\partial}{\partial s} \hat{w}_{\delta} + \cfrac{1}{\delta_{0}^{2}(1+\xi 
\delta_{0}\kappa)}\cfrac{\partial}{\partial\xi}\cfrac{1}{(1+\xi \delta_{0} \kappa)} \cfrac{\partial}{\partial \xi} \hat{w}_{\delta} + k^2_\delta n \hat{w}_{\delta} = 0 \;\text{in}\; 
\mathcal{G}
\end{equation}
together with the boundary conditions
\[
\begin{cases}
 \hat{w}_{\delta}(s,g(s))= 0 &\quad s \in [0,s_{0}], \\
 \hat{w}_{\delta}(s,0)= \wt{v}_{\delta}(s,0)&\quad  s \in [0,s_{0}],\\
\displaystyle{\cfrac{1}{\delta_{0}} \D{\hat{w}_{\delta}}{\xi}|_{\xi=0} =\D{\wt{v}_{\delta}}{\eta}|_{\eta=0}}&\quad s \in [0,s_{0}]\\
\end{cases}
\]
where $\wt{v}_\delta$ is defined by (\ref{tilde}). Let us multiply 
(\ref{EqLocal}) by $\delta_0^2(1+\xi\delta_0\kappa)^3$ to obtain 
\begin{equation*}
\sum_{k=0}^5\delta_0^k A_k  \hat{w}_{\delta} =0
\end{equation*}
where the  $(A_k)_{k=0,\cdots,5}$ are differential operators of order $2$ at maximum with the following expression for the few first terms
\begin{align*}
A_0 &=\D{^2}{\xi^2},\\
A_1 &=3\xi \kappa\D{^2}{\xi^2} +\kappa \D{}{\xi},\\
A_2&= 3\xi^2\kappa^2  \D{^2}{\xi^2} +2 \xi \kappa^2 \D{}{\xi} +   \D{^2}{s^2}+\ld_0 n,\\
A_3 &= \cdots
\end{align*}
Hence, by equating the terms of same order in $\delta$, the function $\hat{w}_{k}$ for $k\in \mathbb {N}$, solves
\begin{eqnarray}
&A_0 \hat{w}_k = - \sum_{l=1}^5 A_l \hat{w}_{k-l}  &\;\; \qquad \text{in}\; \mathcal{G}, \label{eq:volw}\\
& \hat{w}_k(s,g(s)) = 0 & \qquad s \in [0,s_{0}],\label{eq:BCw}\\
&\hat{w}_{k}(s,0) = \wt{v}_k(s,0) & \qquad s \in [0,s_{0}],  \label{eq:BCv}\\
&\displaystyle \D{\hat{w}_k}{\xi}(s,0) = \D{\wt{v}_{k-1}}{\eta}(s,0) & \qquad s \in [0,s_{0}], \label{eq:BCdw}
\end{eqnarray}
with the convention that $\hat{w}_k = v_k=0$ for negative $k$. The functions $v_k$ also satisfy
\begin{equation}
\label{eq:volv}
\Delta v_k + \ld_0 v_k = -\sum_{l=1}^{k}\ld_{l} v_{k-l}.
\end{equation}
Now we can easily obtain the formal expansion at any order by solving \eqref{eq:volw}--\eqref{eq:volv} recursively.
\paragraph*{Order 0.\\}
From \eqref{eq:volw} we have
\[
\frac{\partial^2 \hat{w}_{0}}{\partial^2 \xi} =0  \qquad\text{ in } \mc{G}
\]
and using the boundary conditions (\ref{eq:BCw})  and (\ref{eq:BCdw}) we obtain $\hat{w}_{0}=0$ on $\mathcal{G}$. Equation (\ref{eq:BCv}) together with 
(\ref{eq:volv}) give that $(\ld_{0},v_{0})$ solves
\begin{equation}
\label{EqFirstD}
\begin{cases}
\Delta v_{0} +\ld_{0} v_{0} =0 \text{ in } \Om, \\
v_{0} = 0 \text{ on } \Ga
\end{cases}
\end{equation}
and  hence we define  $(\ld_{0},v_{0})$ as being an eigenpair of the $-\Delta$ in $\Om$ with Dirichlet boundary condition and $\|v_{0}\|_{L^{2}(\Om)}=1$.  We remark  that 
$v_0$ is not uniquely determined (since it can be any Dirichlet eigenfunction), but this will be made precise  later in the convergence analysis.  Nevertheless, we 
assume that $\ld_0$ is simple, which is the case for example for the first Dirichlet eigenvalue of $-\Delta$ in a  Lipschitz and connected domain. The latter 
assumption is necessary to simplify the formal analysis to come.

\paragraph*{Order 1.\\}
Having determined $\hat{w}_0$ and $v_0$ we  iterate the process and obtain that $\hat{w}_{1}$ is the solution to
\[
\frac{\partial^2 \hat{w}_{1}}{\partial \xi^2} =0  \qquad\text{ in } \mc{G}
\]
with boundary conditions \eqref{eq:BCw} and \eqref{eq:BCdw}. That gives
\[
\hat{w}_{1}(s,\xi) =  \D{\wt{v}_{0}}{\eta}(s,0) \xi -g(s) \D{\wt{v}_{0}}{\eta}(s,0) \qquad\text{ in } \mc{G}
\]
The function $v_{1}$ solves
\begin{equation}
\label{Eqv1}
\begin{cases}
\Delta v_{1} + \ld_{0}v_{1} = -\ld_{1}v_{0} \;\text{in} \;\Om, \\
\displaystyle v_{1} = -g  \D{v_{0}}{\ni} \; \text{on}\;\Ga.
\end{cases}
\end{equation}
Since $\ld_{0}$ is a simple eigenvalue for the operator $-\Delta$ with a Dirichlet boundary condition, to ensure uniqueness of $v_1$ we have to constraint $v_1$ 
to be orthogonal to $v_0$ in $L^2(\Om)$. This compatibility condition gives a unique definition for $\ld_1$. By multiplying the first equation of 
\eqref{Eqv1} by $v_0$ and by integrating by part we obtain 
\begin{align}
\ld_{1} &= \displaystyle\int_{\Ga} g \left|\D{v_{0}}{\ni}\right|^{2}\,ds. \label{defl1}
\end{align}
Here we see the simplification due to the assumption that $\ld_0$ is simple. If this does not hold, then the definition of $\ld_1$ does not seem to be obvious.
\paragraph*{Order 2.\\}
To obtain the next term in the asymptotic expansion we iterate the process once more, which yields to the following  equation for $\hat{w}_{2}$
\[
\D{^2\hat{w}_{2}}{\xi^2} +\kappa \D{\hat{w}_{1}}{\xi}=0 \qquad\text{ in } \mc{G}.
\]
The boundary conditions \eqref{eq:BCw} and \eqref{eq:BCdw} on $\Ga$ and $\Ga_{\delta}$ respectively, imply that $\hat{w}_{2}$  is given by
\[
\hat{w}_{2}(s,\xi) = -\frac{\kappa(s)}{2}\D{\wt v_{0}}{\eta}(s,0) \,\xi^{2} + \D{\wt v_{1}}{\eta}(s,0) \, \xi +\frac{ \kappa(s)}{2}   \D{\wt v_{0}}{\eta}(s,0) g(s)^2 - \D{\wt v_{1}}{\eta}
(s,0)g(s) \qquad\text{ in } \mc{G}
\]
where $\kappa$ is the curvature defined in Section \ref{prelim}. From this we deduce that $v_2$ solves
\begin{equation}
\label{eq:v2}
\begin{cases}
\Delta v_2 + \ld_0 v_2 = -\ld_1v_1 -\ld_2 v_0 \text{ in } \Om,\\
\displaystyle v_2 =\frac{ \kappa}{2}  \D{v_{0}}{\ni} g^{2} -  \D{v_{1}}{\ni}g \text{ on } \Ga.
\end{cases}
\end{equation}
Once more, we have to constraint $v_2$ to be orthogonal to $v_0$ and this uniquely defines $\ld_2$ as being
\begin{align}
\ld_2 = & - \int_\Ga \left(  \frac{\kappa}{2}  \D{v_{0}}{\ni} g^{2} - \D{v_{1}}{\ni}g \right) \D{v_0}{\ni}\,ds. \label{ldef2}
\end{align}

\paragraph*{Order k.\\}
Now it becomes clear how to recursively obtain each of the terms in the asymptotic expansion. In particular,  for $k>1$ we assume that the functions $
\hat{w}_l$ and $v_l$ as well as the real numbers $\ld_l$ are well defined for $l<k$. Assume moreover that for all $0<l<k$, 
\begin{equation*}
\int_\Om v_l v_0dx =0
\end{equation*}
and that $\|v_0\|_{L^2(\Om)}=1$.
The first step consists in computing $\hat{w}_k$ by solving \eqref{eq:volw} together with the boundary conditions \eqref{eq:BCw} and \eqref{eq:BCdw}. This 
uniquely determines  $\hat{w}_k$ which leads to an explicit formula for $\hat{w}_k$. Then, by \eqref{eq:BCv} and \eqref{eq:volv}, $v_k$ is uniquely defined as 
being the solution to
\[
\begin{cases}
\Delta v_k + \ld_0 v_k = -\sum_{l=1}^{k} \ld_l v_{k-l}  \text{ in } \Om,\\
v_k = w_k \text{ on } \Ga, \\
\displaystyle\int_\Om v_kv_0 \, ds =0
\end{cases}
\]
where the last equation uniquely defines $\ld_k$ as being
\begin{align*}
\ld_k &= - \int_\Ga w_k \D{v_0}{\ni}ds.
\end{align*}

Of course the asymptotic expansion obtained above is only formal at this point. The next section is dedicated to its convergence analysis.
\section{Convergence analysis}
Our main goal in this section is to rigorously justify the asymptotic expansion formally obtained in the previous section. To this end, for sake of simplicity of 
presentation and to avoid secondary technical difficulties we assume that  the thickness of the thin layer is constant (i.e. $g\equiv 1$), that $n$ is in $C^
\infty(\b{\Om}\setminus\Om_\delta)$ and that $\Ga$ is of class $C^\infty$ as well. Moreover, we only perform the convergence analysis for the first transmission 
eigenvalue that we denote $\ld^1_\delta := (k^1_\delta)^2$. More specifically, in the following we justify the expansion 
\begin{equation}\label{main-asymp}
\ld_\delta^1=\ld_0 + \delta \ld_1 + \delta^2 \ld_2 + \O(\delta^3) 
\end{equation}
where $\lambda_0$ is the first Dirichlet eigenvalue for the $-\Delta$ operator in $\Om$, $\lambda_1$ and $\lambda_2$ are given in the previous section and $
\O(x)$ stands for a generic function in  $C^\infty(\bR^+)$ such that  
\[
 |\O(x)| \leq C|x|
\]
for some constant $C>0$ independent of  $x \in \bR^+$.
The main ingredient to arrive at  such a result is to establish explicit a priori estimates with respect to $\delta$ for the solutions of the interior transmission problem
\begin{equation*}
\begin{cases}
\Delta w_\delta=f_1 \text{ in }  \lay, \\
\Delta v_\delta =f_2 \text{ in } \Om, \\
\displaystyle  \D{v_\delta}{\nnn} - \D{w_\delta}{\nnn}=f_3 \,, \quad v_\delta-w_\delta = f_4 \text{ on }  \Ga,\\
w_\delta =0 \text{ on } \Ga_\delta.
\end{cases}
\end{equation*}
These stability estimates  are stated in the two following propositions. The proof of these propositions requires a few technical lemmas which we state and prove in 
Appendix in order to maintain the main focus of this paper.
\begin{proposition}
\label{prop:generalbound}
Let $v_\delta \in H^2(\Om)$ and $w_\delta \in H^2(\lay)$ be such that for some $s\geq0$
\begin{eqnarray}
&&\|\Delta w_\delta\|_{L^2(\lay)} \leq \O (\delta^{s}) \label{eq:b1} \\
&&w_\delta =0 \text{ on } \Ga_\delta \notag\\
&&\|\Delta v_\delta\|_{L^2(\Om)} \leq \O(\delta^{s}) \label{eq:bd3}\\
&&\left\| \D{v_\delta}{\ni} - \D{w_\delta}{\ni} \right\|_{H^{1/2}(\Ga)} \leq \O(\delta^s) \label{eq:bd4}\\
&&\left\| v_\delta - w_\delta  \right\|_{H^{3/2}(\Ga)} \leq \O(\delta^s).
\end{eqnarray}
 Then, for sufficiently small $\delta >0$,
\[
\|w_{\delta}\|_{H^{2}(\lay)} \leq \O(\delta^s) \text{ and
 } \|w_{\delta}\|_{L^{2}(\lay)} \leq \O(\delta^{s+1}).
\]
\end{proposition}
\begin{proof}
First we prove by contradiction that $\|w_{\delta}\|_{H^2(\lay)} \leq
\O(\delta^s)$. Assume   to the contrary that the latter is not true, then we
can state that (up to an extracted subsequence)
\[
\ga_\delta := \delta^s \|w_{\delta}\|^{-1}_{H^2(\lay)} \underset{\delta\to 0}{\longrightarrow}  0.
\]
Since $w_{\delta}$ is in $H^2(\lay)$ and since $\left\| v_\delta - w_\delta  \right\|_{H^{3/2}(\Ga)} \leq \O(\delta^s)$, from classic elliptic regularity for the Laplacian 
with Dirichlet boundary condition there exists a constant $C$ independent of $\delta$ such that 
\[
\|v_{\delta}\|_{H^{2}(\Om)} \leq C\left(  \|\Delta v_{\delta}\|_{L^{2}(\Om)}+ \|w_{\delta}\|_{H^{3/2}(\Ga)} + \O(\delta^s)\right)
\]
and by using \eqref{eq:bd3}, Lemma \ref{Le:Trace} and the fact that $\ga_\delta$ is bounded when $\delta \to 0$ we deduce that
\[
\frac{\|v_{\delta}\|_{H^{2}(\Om)}}{ \|w_{\delta}\|_{H^2(\lay)} } \leq \frac{1}{\|w_{\delta}\|_{H^2(\lay)}}\left( C \|w_{\delta}\|_{H^{2}(\lay)}  + \O(\delta^s)\right) \leq C
\]
for $C>0$ independent of $\delta$. Hence by \eqref{eq:bd4}  and the trace theorem we have that
\[
\frac{1}{ \|w_{\delta}\|_{H^2(\lay)} } \left\|\D{w_{\delta}}{\ni}\right\|_{H^{1/2}(\Ga)} \leq C.
\]
Now by applying  Lemma \ref{LeRegInd} and Lemma \ref{LeBoundUnif} to the function $ w_{\delta}/ \|w_{\delta}\|_{H^2(\lay)} $, we have that  there exists a 
constant $C>0$ independent of $\delta$ such that
\[
1  \leq C\left(  \frac{\|\Delta w_{\delta}\|_{L^{2}(\lay)}}{\|w_{\delta}\|_{H^2(\lay)}}+ \ve(\delta)\right)
\]
where $\ve(\delta) \rightarrow 0$ when $\delta$ goes to $0$. Hence by using \eqref{eq:b1} since $\ga_\delta = \delta^s \|w_{\delta}\|_{H^2(\lay)}^{-1}$ goes to $0$ 
when $\delta$ goes to $0$  we obtain
\[
1 \leq \ve(\delta) \quad \underset{\delta \to 0}{\longrightarrow 0}.
\]
Thus the first estimate of the lemma holds, and then an application of \eqref{eq:TraceVol} together with \eqref{eq:TraceBord} imply the second estimate.
\end{proof}

\begin{proposition}
\label{prop:wH1}
Let $w_\delta \in H^2(\lay)$ be such that for some $s\geq0$ we have
\[
\|\Delta w_\delta \|_{L^2(\lay)} \leq \O( \delta^s)
\]
and 
\[
w_\delta =0 \text{ on } \Ga_\delta \ , \quad \left\|\D{w_\delta}{\ni} \right\|_{H^{-1/2}(\Ga)} \leq \O( \delta^{s+1/2}).
\]
Then, for sufficiently small $\delta$,
\[
\|w_\delta\|_{H^1(\lay)} \leq \O( \delta^{s+1/2}) \text{ and }
\|w_\delta\|_{L^2(\lay)} \leq \O( \delta^{s+1}).
\]
\end{proposition}
\begin{proof}
First by  \eqref{eq:TraceVol} and  Lemma \ref{Le:Trace} there exists $C>0$ such that for sufficiently small $\delta>0$  we have 
\begin{equation}
\label{eq:wdeltal2h1}
\|w_\delta\|_{L^2(\lay)} \leq C \delta^{1/2}\|w_\delta\|_{H^1(\lay)}.
\end{equation}
An application of  Green's identity in  $\lay$  yields
\[
\int_{\lay} |\nabla w_\delta|^{2}\, dx =\int_{\lay} -\Delta w_\delta w_\delta\, dx - \int_{\Ga} \D{w_\delta}{\ni}w_\delta\,ds, 
\]
and using the assumptions of the lemma and \eqref{eq:wdeltal2h1}  we finally obtain 
\[
\int_{\lay} |\nabla w_\delta|^{2} \leq \O( \delta^{s+1/2}) \|w_\delta\|_{H^1(\lay)}
\]
which proves the first estimate. We obtain the second estimate by simply using \eqref{eq:wdeltal2h1}.
\end{proof}

Now with help of  the above propositions we are able to prove the convergence  of our asymptotic expansion.

\subsection{The convergence of the zero order approximation }
We start with the convergence of the zero order term  in the expansion, which can be easily obtained from the expression satisfied by  the first transmission 
eigenvalue.
\begin{theorem}
\label{ThConvEigenvalue1}
Let $\ld^{1}_{\delta}$ be the first real interior transmission eigenvalue, then
for sufficiently small $\delta >0$,
\[
 \ld_\delta^1 = \ld_0 + \O(\delta).
\]
\end{theorem}
\begin{proof}
We first observe that for $\ld< \ld_0$, where $\ld_0$ is the first Dirichlet  eigenvalue for $-\Delta$ in $\Om$, the operator $A_{\ld}  - \ld^{2} B$ defined in Section \ref{sec:description}
(where we set $\lambda:=k^2$) is injective. Indeed for all $u \in W_\delta$,
\begin{align*}
((A_{\ld}  - \ld^{2} B)u,u)_{W_\delta} &= \int_{\lay} \frac{1}{1-n} |\Delta u + \ld u| dx - \ld^2 \int_\Om |u|^2 dx + \ld \int_\Om| \nabla u| ^2 dx\\
&\geq \ld \|\nabla u\|^2_{L^2(\Om)} \left( 1 - \frac{\ld}{\ld_0}\right)
\end{align*}
where we have used  the Poincar{\'e} inequality. Hence we necessarily have
\begin{equation}\label{lb}
\ld_0 \leq\ld_{\delta}^{1}.
\end{equation}
On the other hand it is possible to characterize $\ld_{\delta}^{1}$ via the Max-Min principle (see \cite{CaCoHa12} for details) as 
\begin{align*}
2(\ld_{\delta}^{1})^{2} = &\inf_{\small \begin{array}{rrcll} u\in W_\delta \;\;\; \\ \|u\|_{L^{2}(\Om)}=1 \end{array}} (A_{\ld^{1}_{\delta}} u,u)_{W_{\delta}} \\
=& \inf_{\small \begin{array}{rrcll} u\in W_\delta \;\;\; \\ \|u\|_{L^{2}(\Om)}=1 \end{array}} \int_{\lay}\frac{1}{1-n} |\Delta u+\ld_{\delta}^{1} u|^{2} \,dx +(\ld_{\delta}
^{1})^{2} + \ld_{\delta}^{1} \int_{\Om} |\nabla u|^{2}\,dx.
\end{align*}
Now if we take $u=u^{1}_{\delta,D}$ where $u^{1}_{\delta,D}$ the first Dirichlet eigenfunction for $-\Delta$ in $\Om_{\delta}$ associated with the eigenvalue $ 
\ld^{1}_{\delta,D}$ extended by $0$ outside $\Om_{\delta}$  such that $\|u^{1}_{\delta,D}\|_{L^{2}(\Om_{\delta})}=1$ (note that due to the zero boundary condition 
on $\Ga_\delta$ the extension by zero of $u^{1}_{\delta,D}$ is in $W_\delta$) we obtain
\[
(\ld_{\delta}^{1})^{2} \leq  \ld_{\delta}^{1}\,\ld^{1}_{\delta,D}
\] 
or equivalently
\begin{equation}\label{ub}
\ld_{\delta}^{1} \leq \ld^{1}_{\delta,D}
\end{equation}
since $\ld_\delta^1$ is bounded below by $\ld_0$. To conclude, we remark that the first Dirichlet eigenvalue for $-\Delta$ is Fr\'echet differentiable  
with respect to the shape (see \cite{Hen06}). A consequence of this result is that there exists a constant $C>0$ such that for all $\delta$ sufficiently small
\[
|\ld_{\delta,D}^{1} - \ld_0|\leq C\delta.
\]
This  fact  together with the lower and upper bounds (\ref{lb}) and (\ref{ub}) respectively, ends the proof.
\end{proof}
\begin{remark}
Theorem \ref{ThConvEigenvalue1} is the cornerstone of our analysis since it allows to uniquely define the Dirichlet eigenvalue which is the first term in the 
asymptotic expansion of the transmission eigenvalue. The other terms in \eqref{main-asymp} are then uniquely defined.
\end{remark}

\subsection{The convergence of the first order approximation }
In order to  proceed with next  order approximation, we must  first  prove some estimates for the first order approximation of  the corresponding eigenfunction. To 
this end let  us define
\begin{equation}\label{eq1}
e^1_w := w_\delta^1 -\delta w_1 \qquad\text{  in } \lay
\end{equation}
extended by $0$ in $\Om_\delta$ and
\begin{equation}\label{eq2}
e^1_v := v_\delta^1 -v_0 \qquad\text{  in } \Om
\end{equation}
where $v_{0}$ is a solution to (\ref{EqFirstD}) such that $\|v_{0}\|_{L^{2}(\Om)}=1$, $w_1(x) := \hat{w}_1(s,\xi)$ with $x=\vp(s,\delta \xi)$, and $(w_\delta^1, v_
\delta^1)$  are the eigenfuctions corresponding to the first transmission eigenvalue $\lambda^1_\delta$. We also extend $w_\delta^1$ by $0$ inside $\Om_\delta
$. We remark that since $w_\delta^1$ and $w_1$ vanish on $\Ga_\delta$, the functions $e^1_w$ and $w_\delta^1$ are continuous across the interface $\Ga_
\delta$. Let us begin with a lemma that provides $\delta$-explicit a priori estimates for  $w_\delta^1$ which will enable us to derive estimates for the first order 
approximation of $w_\delta^1$ and $v_\delta^1$. 
\begin{lemma}
\label{le:wdelta1}
There exists a constant $C>0$ such that for sufficiently small  $\delta>0$ we have
\begin{equation*}
\|w_{\delta}^{1}\|_{H^{2}(\lay)} \leq C \quad \text{and} \quad \|w_{\delta}^{1}\|_{L^{2}(\lay)} \leq \O(\delta).
\end{equation*}
\end{lemma}
\begin{proof}
We show that Proposition \ref{prop:generalbound} applies to $(v_\delta^1,w_\delta^1)$ for $s=0$. To this end, since 
\begin{equation*}
2(\ld_{\delta}^{1})^{2} = (A_{\ld^{1}_{\delta}} u_{\delta}^{1},u_{\delta}^{1})_{W_{\delta}} \\
\end{equation*}
and since $A_{\ld^{1}_{\delta}}$ is coercive with a coercivity constant independent of $\delta$ (see Proposition \ref{prop:Acoercive}), there exists a constant $C$ 
independent of $\delta$ such that 
\begin{equation}
\label{eq:unifborneudelta}
\| u_{\delta}^{1}\|_{H^{1}(\Om)} + \|\Delta  u_{\delta}^{1}\|_{L^{2}(\lay)} \leq C
\end{equation}
where $u_\delta^1$ is defined by (\ref{eq:defu}).
A straightforward consequence of (\ref{eq:unifborneudelta}) is that there exists another constant $C$ still independent of $\delta$ such that 
\begin{equation*}
\|w_{\delta}^{1}\|_{L^{2}(\lay)}\leq C \quad \text{and} \quad \|v_{\delta}^{1}\|_{L^{2}(\Om)}\leq C.
\end{equation*}
Since $-\Delta w_\delta^1 = \ld_\delta^1n w_\delta^1$ and $-\Delta v_\delta^1 = \ld_\delta^1 v_\delta^1$, by using Proposition \ref{prop:generalbound} with $s=0$ 
we have that there exists $C>0$ such that
\[
 \|w_{\delta}^{1}\|_{H^{2}(\lay)} \leq C \quad \text{and} \quad  \|w_{\delta}^{1}\|_{L^{2}(\lay)} \leq  \O(\delta).
\]
which ends the proof.
\end{proof}

\begin{lemma}
\label{prop:ev1}
The following error estimates hold true for sufficiently small  $\delta>0$
\begin{equation}\label{eq:ev1}
\|e_v^1\|_{H^{1}(\Om)} \leq \O( \delta),
\end{equation}
and 
\begin{equation}
\label{eq:ew1}
\|e_w^1\|_{H^{1}(\Om)} \leq  \O( \delta) \ , \quad \|e_w^1\|_{L^2(\Om)} \leq  \O( \delta^{3/2}), \quad \|w^1_\delta\|_{L^2(\Om)} \leq  \O( \delta^{3/2}).
\end{equation}
\end{lemma}

 \begin{proof}
 The idea of the proof is to establish that $v_\delta^{1}$ is a quasi Dirichlet eigenfunction for $-\Delta$ in the domain $\Om$ and then apply Lemma 
\ref{Le:ConvEigen}. From the first estimate of Lemma \ref{le:wdelta1} and  the inequality \eqref{eq:TraceBord} the trace of $w_\delta^1$ on $\Ga$ 
satisfies
 \[
 \|w_{\delta}^{1}\|_{H^{1/2}(\Ga)} \leq \O(\delta^{1/2}).
 \]
Let us define $\theta_{w_{\delta}^{1}}$  a lifting of $w_{\delta}^{1}$ in $H^{1}(\Om)$ such that $\theta_{w_{\delta}^{1}}|_{\Ga}=w^{1}_{\delta}|_{\Ga}$ and  that 
$\|\theta_{w_{\delta}^{1}}\|_{H^{1}(\Om)} \leq \O(\delta^{1/2})$. 
Let us set $\b{v}_{\delta}^{1} := v_{\delta}^{1}-\theta_{w_{\delta}^{1}}$ and show that $\b{v}_{\delta}^{1}$ is close to $v_{0}$. Indeed, for all $\psi \in H^{1}_{0}(\Om)
$ we have that 
\begin{align*}
\left|\int_{\Om} \nabla\b{v}_{\delta}^{1}\cdot \nabla \psi - \ld_0\b{v}_{\delta}^{1}\psi \, dx\right|&= \left|(\ld^{1}_{\delta}-\ld_0) \int_{\Om}v_{\delta}^{1}\psi  \,dx - 
\int_{\Om} ( \nabla\theta_{w_{\delta}^{1}}\cdot \nabla \psi- \ld_0\theta_{w_{\delta}^{1}}\psi ) \, dx \right| \\
&\leq \O(\delta^{1/2}) \|\psi\|_{H^{1}(\Om)}.
\end{align*}
Moreover since $\|u_\delta^1\|_{L^2(\Om)} =1$,  Lemma \ref{le:wdelta1} implies that 
\begin{equation}
\label{eq:Cdelta10}
|\|\b{v}_{\delta}^{1}\|_{L^2(\Om)}-1|\leq\O(\delta^{1/2})
\end{equation}
and hence by virtue of Lemma \ref{Le:ConvEigen} there exists $C>0$  and $v_0$ solution to (\ref{EqFirstD}) with $\|v_0\|_{L^2(\Om)}=1$ such that for all $\delta$ 
sufficiently small we have
\begin{equation}\label{eqq}
\|\b{v}_{\delta}^{1} - v_{0}\|_{H^{1}(\Om)} \leq \O(\delta^{1/2}).
\end{equation}
Note that this last inequality uniquely determines the function $v_0$. From the definition of $\b{v}_{\delta}^{1}$ and the above  bound on the lifting,  (\ref{eqq}) 
takes the form 
\begin{equation}
\label{eq:vv0}
\|e_v^1\|_{H^{1}(\Om)}\leq \O(\delta^{1/2}).
\end{equation}
But since $w_{1}$ solves \eqref{eq:volw} we have
\[
\Delta w_1  = \frac{1}{(1+\delta\xi \kappa)^3}\left(\sum_{k=0}^5 \delta^{k-2} A_k\hat{w}_1 - \ld_\delta^1 n \hat{w}_1\right)
\]
whence
\begin{equation*}
\|\Delta w_1\|_{L^2(\lay)} \leq C( \delta^{-1/2} \|A_1\hat{w}_1\|_{L^2(\mathcal{G})} + \delta^{1/2}\|\hat{w}_1\|_{L^2(\mathcal{G})} \leq C\delta^{-1/2}.
\end{equation*}
Therefore
\begin{equation}
\label{eq:deltaew1}
\|\Delta e_w^1\|_{L^2(\lay)} \leq \O(\delta^{1/2}), 
\end{equation}
and in addition we also  have by \eqref{eq:vv0}
\begin{equation}
\label{eq:deltaev1}
\|\Delta e_v^1\|_{L^2(\Om)} \leq\O(\delta^{1/2}) \ , \quad \left\| \D{e_v^1}{\ni} - \D{e_w^1}{\ni} \right\|_{H^{1/2}(\Ga)} = 0 
\end{equation}
and
\[
\left\| e_v^1 -e_w^1  \right\|_{H^{3/2}(\Ga)}=\left\| \delta w_1  \right\|_{H^{3/2}(\Ga)}\leq \O(\delta).
\]
An application of Proposition \ref{prop:generalbound} now implies
\[
\| e_w^1\|_{H^2(\lay)} \leq  \O(\delta^{1/2})
\]
and then thanks to   \eqref{eq:TraceBord} applied  $e_w^1$, we can improve the bound on $w_\delta^1$ as follows
\[
 \|w_{\delta}^{1}\|_{H^{1/2}(\Ga)} \leq \O(\delta)
\]
since $\|w_1\|_{H^{1/2}(\Ga)}\leq C$ for all $\delta>0$. Now repeating the previous steps of the proof allows us  to  improve the bound on $e_v^1$ as follows
\begin{equation*}
\|e_v^1\|_{H^{1}(\Om)}\leq\O(\delta).
\end{equation*}
Finally, this last inequality together with \eqref{eq:deltaew1} and the fact that 
\[
 \D{e_v^1}{\ni}= \D{e_w^1}{\ni} \text{ on } \Ga
\]
yield the desired bounds for $e_w^1$ thanks to Proposition \ref{prop:wH1}.
\end{proof}

As a consequence of the error estimates derived in Lemma \ref{prop:ev1} we can now obtain  the desired first order convergence result which is  stated in the 
theorem below.
\begin{theorem}
\label{co:eigenvalue2}
The following asymptotic expansion for the first transmission eigenvalue
$\ld_{\delta}^{1}$ holds true for sufficiently small  $\delta>0$,
\[
\ld_{\delta}^{1} = \ld_0+\delta\ld_1+\O(\delta^2),
\]
where $\lambda_0$ is the first Dirichlet eigenvalue for $-\Delta$ in $\Omega$ and $\lambda_1$ is defined by (\ref{defl1}).
\end{theorem}
 
\begin{proof}
Let  $u_{\delta}^{1}\in W_\delta$ be  defined  by (\ref{eq:defu}) with $v_\delta^1$ and $w_\delta^1$ normalized such that $\|u^1_{\delta}\|_{L^{2}(\Om)}=1$. Then 
from \cite{CaCoHa12} we have that
\[
(\ld_{\delta}^{1})^2 =\int_{\lay}\frac{1}{1-n} |\Delta u_{\delta}^{1}+\ld_{\delta}^{1} u_{\delta}^{1}|^{2} \,dx + \ld_{\delta}^{1} \int_{\Om} |\nabla u_{\delta}^{1}|^{2}\,dx.
\]
and using the definition (\ref{eq:defu}) of $u_{\delta}^{1}$ and the equations for $v_{\delta}^{1}$  this becomes
\begin{equation}
 \label{eq:charald}
\ld_{\delta}^{1} =\int_{\lay}\ld_{1}^{\delta}(1-n) |w_{\delta}^{1}|^{2} \,dx +  \int_{\Om} |\nabla u_{\delta}^{1}|^{2}\,dx.
\end{equation}
From (\ref{eq:ew1}) the first term in (\ref{eq:charald}) is of order $\delta^{3}$, hence we need to develop only the 
second term. To this end we can write
\begin{align}
\int_{\Om} |\nabla u_{\delta}^{1}|^{2}\,dx &= \int_{\Om} |\nabla (e_w^1- e_v^1  - v_{0} + \delta w_{1})|^{2}\,dx\label{udest}\\
&=\int_{\Om}  |\nabla e_w^{1}|^{2} \, dx + \int_{\Om}  |\nabla e_v^{1}|^{2} \, dx + \int_{\Om} |\nabla v_{0}|^{2} \,dx+ \delta^{2} \int_{\lay} |\nabla w_{1}|^{2} \, dx\nonumber 
\\&-2   \int_{\Om} \nabla v_{0} \nabla (w_\delta^1-e_v^{1} ) \, dx + 2 \delta  \int_{\Om} \nabla (e_w^{1}-e_v^1) \nabla w_{1} \, dx -2 \int_{\Om} \nabla e_v^{1} \nabla 
e_w^1 \, dx.\nonumber
\end{align}
Recall that from Lemma  \ref{prop:ev1}  we have that 
\[
 \int_{\Om}  |\nabla e_v^{1}|^{2} \, dx \leq \O(\delta^2), \quad  \int_{\Om}  |\nabla e_w^{1}|^{2} \, dx \leq \O(\delta^2) \text{ and } \int_{\Om} \nabla e_v^{1} \nabla 
e_w^1 \, dx \leq \O(\delta^2).
\]
Furthermore, from the definition of $v_0$ we have that
\[
\int_{\Om} |\nabla v_{0}|^{2} \, dx = \ld_0,
\]
and from the definition of $w_1$ we have that 
\[
\delta^{2} \int_{\lay} |\nabla w_{1}|^{2} \, dx = \delta^{2} \int_{0}^{s_0}\left(\left|\D{v_{0}}{\ni}\right|^{2} \int_{0}^{\delta} \cfrac{1}{\delta^{2}} J_{s,\eta} \, d\eta\right) ds + 
\O(\delta^{2}) = \delta\ld_1 + \O(\delta^{2}).
\]
In addition we also know that $\|u_{\delta}^{1}\|_{L^{2}(\Om)}=1$ that is
\[
1 = \int_{\Om} |u_{\delta}^{1}|^{2} \, dx = \int_{\Om} |u_{\delta}^{1}+v_{0}|^{2} - 2(u_{\delta}^{1} + v_{0}) v_{0} +| v_{0}|^{2 }\, dx.
\]
The estimates of Lemma  \ref{prop:ev1} together with the definitions (\ref{eq1}) and (\ref{eq2}) give that 
$$\|u_{\delta}^{1}+v_{0}\|_{L^2(\Om)}\leq \O(\delta)$$ and hence since $\|v_{0}\|_{L^{2}(\Om)}=1$ we have
\begin{equation}
\label{EqNorm1}
\left|\int_{\Om} v_{0}(u_{\delta}^1+v_{0}) \, dx  \right| \leq \O(\delta^2).
\end{equation}
Recalling that $v_{0}$ is a Dirichlet eigenfunction for $-\Delta$ with corresponding eigenvalue $\lambda_0$,  from \eqref{EqNorm1} we now obtain 
\begin{align*}
\left| \int_{\Om} \nabla v_{0} \nabla (w_\delta^{1}-e_v^1) \, dx \right|&=\left| \int_{\Om} \nabla v_{0} \nabla (u_{\delta} + v_{0}) \, dx  \right| \\
&= \ld_0\left| \int_{\Om} v_{0} (u_{\delta} + v_{0}) \, dx \right| \leq \O(\delta^2).
\end{align*}
Finally, noting that 
\[
\int_{\lay}\nabla (e_w^{1}- e_v^1) \nabla w_{1} \, dx  = \int_{\lay}-\Delta (e_w^{1}-e_v^1) w_{1} \, dx - \int_{\Ga} \D{}{\ni}(e_w^1-e_v^1) \, w_{1} \, ds
\]
and using $ \D{e_w^1}{\ni} = \D{e_v^1}{\ni}$ we obtain
\begin{equation}
\label{Eqe1}
\int_{\lay}\nabla (e_w^{1}- e_v^1) \nabla w_{1} \, dx  = \int_{\lay}-\Delta  (e_w^{1}- e_v^1) w_{1} \, dx.
\end{equation}
But,   (\ref{eq:deltaew1}) and \eqref{eq:deltaev1} imply 
\begin{align*}
\|\Delta (e_w^{1}- e_v^1)\|_{L^{2}(\lay)} &\leq C\delta^{1/2}.
\end{align*}
 This last estimate together with equality (\ref{Eqe1}) and $\|w_1\|_{L^2(\lay)} \leq C \delta^{1/2}$ gives
\[
 \int_{\lay} \nabla (e_w^1-e_v^{1}) \nabla w_{1} \, dx  \leq C\delta.
  \]
We have estimated all the terms in (\ref{udest}) and thus the expression  \eqref{eq:charald} for $\ld^{1}_{\delta}$ finally yields the estimate
\[
\ld_{\delta}^{1} = \ld_{0} +\delta \ld_{1} +\O(\delta^2).
\]
\end{proof}

\subsection{The convergence of the second order approximation}\label{order2}
The goal of this section is twofold. We complete the rigorous justification of the asymptotic expansion (\ref{main-asymp}), and present  a constructive procedure 
how to iteratively obtain the converges of any order in the asymptotic expansion of transmission eigenvalues. To this end, before proceeding with  the 
convergence of the eigenvalues we need to improve the rate of convergence of the corresponding eigenfunctions. This is possible by adding a correction term to 
the eigenfunctions  $v_\delta^1$ and $w_\delta^1$. Let us consider the following error functions
\[
e^2_w := w_\delta^1 -C_\delta^1(\delta w_1 +\delta^2w_2) \qquad\text{  in } \lay
\]
extended by $0$ in $\Om_\delta$ and
\[
e^2_v := v_\delta^1 -C_\delta^1(v_0 +\delta v_1) \qquad\text{  in } \Om
\]
where $w_2(x):=\hat{w}_2(s,\xi)$ and   $v_1$ are defined in Section \ref{sec:devas} and $C_\delta^1:=\|u_\delta^1+\delta v_1\|_{L^2(\Om)}$. As before, the error 
$e^2_w$ is continuous across the interface $\Ga_\delta$ since it vanishes on $\Ga_\delta$. We remark that since $\|u_\delta^1\|_{L^2(\Om)} =1$ 
we have that 
\begin{equation}
\label{eq:Cdelta1}
C_\delta^1 = 1+ \O(\delta).
\end{equation}
 We now proceed as in Lemma \ref{prop:ev1}  to to give a first estimate on $e_w^2$ which on its turn  provides a $H^1$ bound for $e_v^2$ and then iterate the 
procedure to obtain the optimal bounds for $e_w^2$ and $e_v^2$.
 \begin{lemma}
\label{le:ew2}
The following a priori estimates hold for $\delta >0$ sufficiently small:
\[
\|e^2_w\|_{H^2(\lay)} \leq \O(\delta) \quad  \text{ and } \quad \|e^2_w\|_{L^2(\lay)} \leq \O (\delta^2).
\]
\end{lemma}
\begin{proof}
First of all, from the definition of $w_1$ and $w_2$ we have for $\delta >0$ sufficiently small
\begin{equation}
\label{eq:w1w2L2}
\|w_1\|_{L^2(\lay)} \leq \O(\delta^{1/2}),\quad\|w_2\|_{L^2(\lay)} \leq \O(\delta^{1/2}).
\end{equation}
Moreover, from Section \ref{sec:devas}
\[
\Delta(w_1 + \delta w_2) = \frac{1}{(1+\delta\xi \kappa)^3}\left(\sum_{k=0}^5 \delta^{k-2} A_k(\hat{w}_1 +\delta \hat{w}_2)- \ld_\delta^1 n (\hat{w}_1+\delta \hat{w}
_2)\right)
\]
and since  $A_0 \hat{w}_1 =0$ and $A_0 \hat{w}_2 + A_1 \hat{w}_1 =0$ we now have 
\[
\Delta(w_1 + \delta w_2)=  \frac{1}{(1+\delta\xi \kappa)^3}\left(A_1 \hat{w}_2- \ld_\delta^1n (\hat{w}_1+\delta \hat{w}_2) + \sum_{k=2}^5 \delta^{k-2} A_k(\hat{w}_1 
+\delta \hat{w}_2)\right)
\]
which yields in view of \eqref{eq:ew1} and \eqref{eq:w1w2L2}
\begin{equation}
\label{eq:deltaw2}
\|\Delta e_w^2\|_{L^2(\lay)} \leq \O(\delta^{3/2}).
\end{equation}
Next, 
\[
\|e_w^2 - e^2_v\|_{H^{3/2}(\Ga)} = \|C_\delta^1\delta^2 w_2\|_{H^{3/2}(\Ga)} \leq \O( \delta^{2}),
\]
\[
\D{e^2_w}{\ni} = \D{e^2_v}{\ni} \quad \text{ on } \Ga,
\]
and
\begin{equation*}
\|\Delta e_v^2\|_{L^2(\Om)} \leq \O(\delta)
\end{equation*}
whence by applying Lemma \ref{prop:generalbound} with $s=1$ we obtain the result.
\end{proof}

Similarly to the previous section we are now able to obtain convergence rates for the eigenfunctions.
\begin{lemma}
\label{prop:ev2}
The following error estimates hold for $\delta>0$ sufficiently small
\[
\|e^2_v\|_{H^1(\Om)} \leq C \delta^{3/2},
\]
and
\[
\|e_w^2\|_{H^{1}(\Om)} \leq  \O( \delta^{3/2}) \ , \quad \|e_w^2\|_{L^2(\Om)} \leq  \O( \delta^{5/2}).
\]
\end{lemma}
\begin{proof}
We proceed exactly as in the proof of Proposition \ref{prop:ev1}.  To this end, let us define $\d{v}_2:= (C_\delta^1)^{-1}(v_\delta^1 - \delta v_1)$ and then since 
$C_\delta^1 = 1 + \O(\delta)$, we have that 
\begin{align*}
\d{v}_2|_\Ga &=(C_\delta^1)^{-1}( v^1_\delta|_\Ga - \delta v_1|_\Ga -\delta^2 v_2|_\Ga) + (C_\delta^1)^{-1}\delta^2v_2|_\Ga\\
&=(C_\delta^1)^{-1}( v^1_\delta|_\Ga - C^1_\delta( \delta v_1|_\Ga -\delta^2 v_2|_\Ga)) + \O(\delta^2) \\
& =(C_\delta^1)^{-1}e_w^2|_\Ga + \O(\delta^2).  
\end{align*}
Using (\ref{eq:TraceBord}) together with Lemma \ref{le:ew2} we see that
\begin{equation*}
\|\d{v}_2\|_{H^{1/2}(\Ga)} \leq \O(\delta^{1/2}) \| e_w^2\|_{H^2(\lay)} + \O(\delta^2) \leq \O(\delta^{3/2}).
\end{equation*}
Let us denote by $\theta_{\d{v}_2}$ a lifting of $\d{v}_2$ in $\Om$ such that $\theta_{\d{v}_2}|_\Ga = \d{v}_2|_\Ga$ and $\|\theta_{\d{v}_2}\|_{H^1(\Om)} \leq \O(
\delta^{3/2})$ then consider $\b{v}_2 := \d{v}_2 - \theta_{\d{v}_2}$. Obviously $\b{v}_2 \in H^1_0(\Om)$ and for all $
\psi \in H^1_0(\Om)$ we have
\begin{align}
\int_{\Om} \nabla\b{v}_2\cdot \nabla \psi- \ld_0\b{v}_2\psi \, dx&= \int_{\Om}\nabla \d{v}_2\cdot\nabla\psi - \ld_0\d{v}_2\psi  \,dx  \nonumber\\
&- \int_{\Om} \nabla\theta_{\d{v}_2}\cdot \nabla \psi+ \ld_0\theta_{\d{v}_2}\psi \, dx  \label{besi}.
\end{align}
We can estimate the second term easily by using the bound on the lifting and the Cauchy-Schwarz inequality
\[
\left| \int_{\Om} \nabla\theta_{\d{v}_2}\cdot \nabla \psi+ \ld_0\theta_{\d{v}_2}\psi \, dx  \right| \leq \O(\delta^{3/2}) \|\psi\|_{H^1(\Om)}.
\]
Next we consider the first term in (\ref{besi}) containing  $\d{v}_2$.  For all $\psi \in H^1_0(\Om)$  from Theorem \ref{co:eigenvalue2} and Lemma \ref{prop:ev1} we 
can write
\begin{align*}
\int_{\Om}\nabla (v_\delta^1 -\delta v_1)\cdot\nabla\psi \, dx &= \ld_\delta^1\int_\Om v_\delta^1 \psi \,dx - \delta\ld_0\int_\Om v_1 \psi \,dx-\delta \ld_1  \int_\Om v_0 
\psi \,dx \\
& \hspace*{-1cm}=( \ld_0 +\delta \ld_1)\int_\Om v_\delta^1 \psi \,dx - \delta\ld_0\int_\Om v_1 \psi \,dx-\delta \ld_1  \int_\Om v_0 \psi \,dx + \O(\delta^2) \|\psi\|
_{H^1(\Om)}\\
&\hspace*{-1cm} =\ld_0 \int_\Om (v_\delta^1-\delta v_1) \psi \,dx +\O(\delta^2) \|\psi\|_{H^1(\Om)}.
\end{align*}
Next by using \eqref{eq:Cdelta1} we  obtain that there exists another constant $C$ still independent of $\delta$ such that 
\begin{equation*}
\left|\int_{\Om}(\nabla \d{v}_2 \cdot\nabla \psi -\ld_0 \d{v}_2 \psi)\, dx \right| \leq C \delta^{2} \|\psi\|_{H^1(\Om)}
\end{equation*}
for all $\psi \in H^1_0(\Om)$. We can now apply Lemma \ref{Le:ConvEigen} to $\b{v}_2$  since by \eqref{eq:ew1} and the definition of $C^1_\delta$, we have
\[
\|\d{v}_2\|_{L^2(\Om)}=(C_\delta^1)^{-1} \|u_\delta^1+\delta v_1\|_{L^2(\Om)}+\O(\delta^{3/2}) = 1 +\O(\delta^{3/2})
\]
 to obtain 
\begin{equation*}
\left\|\b{v_2}-v_0\right\|_{H^1(\Om)} \leq \O(\delta^{3/2}).
\end{equation*}
From the bound on the lifting $\theta_{\d{v}_2}$  and \eqref{eq:Cdelta1} the latter becomes
\begin{equation*}
\|v_\delta^1 - C_\delta^1v_0 -C_\delta^1\delta v_1\|_{H^1(\Om)}\leq O(\delta^{3/2}).
\end{equation*}
Now it is clear that we can improve the bound on $e_w^2$ by applying Lemma \ref{prop:generalbound} with $s=3/2$ , since 
\begin{equation}
\label{eq:deltaev2}
\|\Delta e_v^2\|_{L^2(\Om)} \leq \O(\delta^{3/2}).
\end{equation}
Therefore we arrive at the following improved estimate 
\[
\|e_w^2\|_{H^2(\lay)} \leq \O(\delta^{3/2}), \qquad \|e_w^2\|_{L^2(\lay)} \leq  \O( \delta^{5/2}).
\]
\end{proof}
\begin{remark}
As in Lemma \ref{prop:ev1} we can improve the bound on $e_v^2$ if we choose 
\[
C_\delta^1 = \|u_\delta^1 +\delta v_1 -\delta w_1\|_{L^2(\Om)}.
\]
Indeed, in this case
\[
\|\b{v}_2\|_{L^2(\Om)} = 1 + \O(\delta^2)
\]
and since $\|e_w^2\|_{H^2(\lay)} \leq \O(\delta^{3/2})$ we deduce that $\|e_v^2\|_{H^{1/2}(\Ga)} \leq \O(\delta^{2})$ which implies that
\[
\|e_v^2\|_{H^1(\Om)} \leq \O(\delta^{2}).
\]
\end{remark}

The estimates obtained in Lemma \ref{prop:ev2}  lead to the following result.
\begin{theorem}
The following expansion for the first transmission eigenvalue holds true for
$\delta>0$ sufficiently small
 \[
\ld^1_\delta = \ld_0 + \delta \ld_1 +\delta^2\ld_2 + \O(\delta^{3}),
 \]
 where $\lambda_0$ is the first Dirichlet eigenvalue for $-\Delta$ in $\Omega$,
 and $\lambda_1$ and $\lambda_2$ are  defined by (\ref{defl1}) and
 (\ref{ldef2}) respectively.
\end{theorem}
\begin{proof}
 Similarly to the proof of Theorem \ref{co:eigenvalue2} we expand the definition of $\ld_\delta^1$ by using the approximate eigenfunctions 
 \[
 w_{{\rm app}}^2 := C_\delta^1(\delta w_1 +\delta^2 w_2), \quad \mbox{and}\quad   v_{{\rm app}}^2 := C_\delta^1(v_0+\delta v_1)
 \] 
 and we extend $ w_{{\rm app}}^2$ by $0$ inside $\Om_\delta$. From the characterization \eqref{eq:charald} of $\ld_\delta^1$ and the bound \eqref{eq:ew1} we 
have
 \[
  \ld_\delta^1 = \int_{\Om} |\nabla u_\delta^1|^2\,dx + \O(\delta^3).
 \]
This writes
\begin{align}
\int_{\Om} |\nabla u_\delta^1|^{2}\,dx &= \int_{\Om} |\nabla (e_w^2- e_v^2  - \vapp^2 + \wapp^2)|^{2}\,dx \label{pse}\\
&=\int_{\Om}  |\nabla e_w^{2}|^{2} \, dx + \int_{\Om}  |\nabla e_v^{2}|^{2} \, dx + \int_{\Om} |\nabla \vapp^2|^{2} \,dx + \int_{\lay} |\nabla \wapp^2|^{2} \, dx\nonumber\\
&-2   \int_{\Om} \nabla \vapp^2 \nabla (w_\delta^1-e_v^{2} ) \, dx + 2   \int_{\lay} \nabla (e_w^{2}-e_v^2) \nabla \wapp^2 \, dx -2 \int_{\Om} \nabla e_v^{2} \nabla 
e_w^2 \, dx.\nonumber
\end{align}
There are several terms to evaluate which we consider one by one in the following.

\medskip

\noindent{\it Step1: Computation of the $\O(\delta^3)$ terms.} From Lemma \ref{prop:ev2} we can see easily
\[
\int_{\Om}  |\nabla e_w^{2}|^{2} \, dx= \O(\delta^3),\quad \int_{\Om}  |\nabla e_v^{2}|^{2} \, dx = \O(\delta^3) \; \text{ and }\; \int_{\Om} \nabla e_v^{2} \nabla e_w^2 \, dx 
=\O(\delta^3).
\]
Furthermore, 
\[
 \int_{\lay} \nabla (e_w^{2}-e_v^2) \nabla \wapp^2 \, dx = -  \int_{\lay} \Delta (e_w^{2}-e_v^2) \wapp^2 \, dx - \int_\Ga \D{}{\ni}(e_w^{2}-e_v^2) \wapp^2 \, dx
\]
and recalling that $\D{e_v^2}{\ni} = \D{e_w^2}{\ni}$ we obtain
\[
 \int_{\lay} \nabla (e_w^{2}-e_v^2) \nabla \wapp^2 \, dx = -  \int_{\lay} \Delta (e_w^{2}-e_v^2) \wapp^2 \, dx.
\]
But \eqref{eq:deltaw2} and \eqref{eq:deltaev2} give
\[
\|\Delta (e_w^{2}-e_v^2)\|_{L^2(\lay)} \leq \O(\delta^{3/2})
\]
which complemented with \eqref{eq:w1w2L2} gives
 \[
 \left|\int_{\lay} \nabla (e_w^{2}-e_v^2) \nabla \wapp^2 \, dx \right| \leq \O(\delta^{3}).
 \]
 
 \medskip
\noindent {\it Step 2: Computation of  $\int_{\Om} |\nabla \vapp^2|^{2} \, dx$.}  From its definition we   have after integration by part:
 \begin{align*}
 \int_{\Om} |\nabla \vapp^2|^{2} \, dx &= (C_\delta^1)^2\left(\int_{\Om} |\nabla v_0|^{2} \, dx + 2\delta \int_\Om \nabla v_0 \cdot \nabla v_1\, dx + \delta^2 \int_\Om |
\nabla v_1|^2\, dx \right)\\
 &= (C_\delta^1)^2\left( \ld_0 + 2\delta \ld_0 \int_\Om v_0 v_1\, dx +2\delta \ld_1  \right.\\
 &\qquad  \left.+\delta^2 \int_\Om (\ld_0 |v_1|^2 + \ld_1 v_0v_1)\, dx + \delta^2 \int_\Ga \D{v_1}{\ni} \D{v_0}{\ni}\,ds  \right) \\
 & = \ld_0\|\vapp^2\|^2_{L^2(\Om)}   + 2(C_\delta^1)^2  \delta \ld_1 +\delta^2  \int_\Ga \D{v_1}{\ni} \D{v_0}{\ni}\,ds.
 \end{align*}
 
 \medskip

\noindent{\it Step 3: Computation of  $\int_{\lay} |\nabla \wapp^2|^{2} \, dx$.} To this end we first write 
\begin{eqnarray}
 \int_{\lay} |\nabla \wapp^2|^{2} \, dx &=& (C_\delta^1)^2 \delta^2 \int_{\lay} |\nabla w_1|^2 \, dx \label{mut} \\
 &+& 2(C_\delta^1)^2 \delta^3 \int_{\lay} \nabla w_1 \cdot \nabla w_2\, dx + (C_\delta^1)^2 \delta^4 \int_{\lay} |\nabla w_2|^2 \, dx. \nonumber
 \end{eqnarray} 
 From the definition of $w_1$ and by using local coordinates we have
 \begin{align*}
 \delta^2 \int_{\lay} |\nabla w_1|^2 \, dx& = \delta^3\int_0^{s_0}\int_0^1\frac{1}{\delta^2}\left|\D{\hat{w}_1}{\xi}(s,\xi) \right|^2 (1+\delta\kappa\xi)\,d\xi ds + 
\O(\delta^3) \\
 & = \delta \ld_1 +\delta^2 \frac{\kappa}{2} \ld_1 + \O( \delta^3).
 \end{align*}
Similarly,  using the definition of $w_2$ we have
 \begin{align*}
 \delta^3 \int_{\lay} \nabla w_1 \cdot \nabla w_2\, dx &= \delta^4 \int_0^{s_0}\int_0^1\frac{1}{\delta^2} \D{\hat{w}_1}{\xi}(s,\xi) \D{\hat{w}_2}{\xi}(s,\xi) d\xi ds + 
\O(\delta^3) \\
 & = \delta^2 \int_0^{s_0}\int_0^1\left( -\kappa \D{v_0}{\ni} \xi + \D{v_1}{\ni}\right) \D{v_0}{\ni} \,d\xi ds + \O(\delta^3)\\
 &= - \delta^2 \frac{\kappa}{2} \ld_1 +\delta^2 \int_\Ga\D{v_1}{\ni} \D{v_0}{\ni} \, ds + \O(\delta^3).
 \end{align*}
 For the last term of (\ref{mut}) we simply have
 \[
  \delta^4 \int_{\lay} |\nabla w_2|^2 \, dx \leq \O(\delta^3).
 \]
Next we need to estimate the constant $(C_\delta^1)^2$. Indeed
 \begin{align}
 (C_\delta^1)^2 &= \int_\Om |u_\delta^1 + \delta v_1|^2 \, dx= \int_\Om | u_\delta^1|^2 +2 \delta \int_\Om u_\delta^1 v_1\, dx + \delta^2 \int_\Om |v_1|^2\, dx \notag \\
 &= 1  +  \delta^2 \int_\Om |v_1|^2\, dx + 2 \delta \int_\Om(u_\delta^1+v_0)v_1\,dx = 1 + \O(\delta^2) \label{eq:expandCdelta1}
 \end{align}
 since $\|u_\delta^1 +v_0\|_{L^2(\Om)}\leq \O(\delta)$. Plugging everything into (\ref{mut}) we finally obtain
 \begin{equation*}
  \int_{\lay} |\nabla \wapp^2|^{2} \, dx = \delta \ld_1 +\delta^2\left( 2 \int_\Ga \D{v_1}{\ni}\D{v_0}{\ni} \, ds - \frac{\kappa}{2} \ld_1\right) + \O(\delta^{3}).
 \end{equation*}

 \medskip
 
\noindent{\it Step 4: Computation of  $\int_{\Om} \nabla \vapp^2 \cdot \nabla (w_\delta^1 - e^2_v) \, dx$.} To this end, we make use of the equation satisfied by $\vapp^2$ 
together with the normalization of $u_\delta^1$ to simplify it. First we can write
 \begin{align*}
 \int_{\Om} \nabla \vapp^2 \cdot \nabla (w_\delta^1 - e^2_v) \, dx &= - \int_{\Om} \Delta \vapp^2  (w_\delta^1 - e^2_v) \, dx -\int_\Ga \D{\vapp^2}{\ni} \vapp^2 \,ds \\
&= C_\delta^1 \int_\Om (\ld_0v_0 + \delta \ld_0 v_1 + \delta \ld_1 v_0) (w_\delta^1 - e_v^2) \, dx \\ &\qquad-(C_\delta^1)^2 \int_\Ga\left(\D{v_0}{\ni} + \delta \D{v_1}
{\ni}\right)\delta v_1 \, ds. 
 \end{align*}
From  \eqref{eq:ew1} and Lemma \ref{prop:ev2} we have that
\[ C_\delta^1 \delta^2 \int_\Om \ld_1 v_1(w_\delta^1 - e_v^2) \, dx = \O(\delta^3)\]
whence
 \begin{equation*}
  \int_{\Om} \nabla \vapp^2 \cdot \nabla (w_\delta^1 - e^2_v) \, dx = (\ld_0+\delta \ld_1) \int_\Om \vapp^2(w_\delta^1 - e_v^2) \, dx +  (C_\delta^1)^2\delta \ld_1 + 
\delta^2\int_\Ga \D{v_1}{\ni} \D{v_0}{\ni} \,ds + \O(\delta^3).
 \end{equation*}
We now use the normalization of $u^1_\delta$ to obtain 
 \begin{align}
 \|u_\delta^1\|^2_{L^2(\Om)} = 1 &= \int_\Om | u_\delta^1 + \vapp^2 - \vapp^2 |^2 \, dx \notag \\
 &= \int_\Om |u_\delta^1 + \vapp^2|^2 \, dx -2 \int_\Om \vapp^2(u_\delta^1 +\vapp^2)\, dx + \int_\Om|\vapp^2|^2 \, dx \notag\\
 & = \O(\delta^3) -2 \int_\Om \vapp^2(u_\delta^1 +\vapp^2)\, dx + \int_\Om|\vapp^2|^2 \, dx \label{eq:normalisationvapp2}.
 \end{align}
 Hence since $u_\delta^1 +\vapp^2 = w_\delta^1 - e_v^2$, we have that 
 \[
 -2  \int_{\Om} \nabla \vapp^2 \cdot \nabla (w_\delta^1 - e^2_v) \, dx = (\ld_0 +\delta \ld_1)\left( 1 - \|\vapp^2\|_{L^2(\Om)}^2  \right) - 2  (C_\delta^1)^2\delta \ld_1 -2 
\delta^2\int_\Ga \D{v_1}{\ni} \D{v_0}{\ni} \,ds + \O(\delta^3).
 \]
The expansion \eqref{eq:expandCdelta1} and the definition of $\vapp^2$ yield
 \[
 \left| 1 - \|\vapp^2\|_{L^2(\Om)}^2  \right| \leq \O(\delta^{2})
 \]
 and consequently
  \[
 -2  \int_{\Om} \nabla \vapp^2 \cdot \nabla (w_\delta^1 - e^2_v) \, dx = \ld_0 \left( 1 - \|\vapp^2\|_{L^2(\Om)}^2  \right) - 2  (C_\delta^1)^2\delta \ld_1 -2 \delta^2\int_
\Om \D{v_1}{\ni} \D{v_0}{\ni} \,ds + \O(\delta^{3}).
 \]
 
 Finally we have all the necessary estimates to reach the conclusion. Plugging the estimates obtained in  Steps 1-4 into (\ref{pse}) leads to the desired final 
estimate:
 \[
 \ld_\delta^1 = \ld_0 + \delta \ld_1 +\delta^2 \ld_2 + \O(\delta^{3}).
 \]
\end{proof}

\begin{remark}
Although we stop at the order two the analysis of Section \ref{order2} is constructive and can in principle be carried through iteratively to any order of 
approximation. Note that  in order to prove the $k$  order of convergence for the transmission eigenvalue we need to prove the $(k-1)$ order of convergence for 
the corresponding eigenfunctions. Also the convergence procedure is not limited to only the first eigenvalue. All these generalizations rely upon the ability to 
compute explicitly  the terms in the asymptotic expansion of the transmission eigenvalues and on the estimate for the zero order approximation of the eigenvalue.
\end{remark}
\begin{remark}
In principle the second order asymptotic expansion of the the first transmission eigenvalue can be used to estimate the thickness of the layer provided that $
\Omega$ is known. In particular
$$\delta\approx\frac{\lambda_\delta^1-\lambda_0}{\lambda_1}$$
where $\lambda_\delta^1$ can be computed from the scattering data (see \cite{CaCoHa10}) and $\lambda_0$ and $\lambda_1$ can be computed.
\end{remark}


\bigskip

{\bf Acknowledgements.}  The  research of F.~C.\ is supported in part by the  Air Force Office of Scientific Research  Grant FA9550-11-1-0189 and NSF Grant 
DMS-1106972. 
\bibliographystyle{rspublicnat}
\bibliography{biblio}

\begin{thebibliography}{15}
\providecommand{\natexlab}[1]{#1}
\expandafter\ifx\csname urlstyle\endcsname\relax
  \providecommand{\doi}[1]{doi:\discretionary{}{}{}#1}\else
  \providecommand{\doi}{doi:\discretionary{}{}{}\begingroup
  \urlstyle{rm}\Url}\fi

\bibitem[{Agmond(1965)}]{Ag}
Agmond, S. 1965 \emph{{Lectures on Elliptic Boundary Value Problems}}.
\newblock Princeton: D. van Nostrand.

\bibitem[{Bendali \& Lemrabet(1996)}]{BenLem96}
Bendali, A. \& Lemrabet, K. 1996 {The effect of a thin coating on the
  scattering of a time-harmonic wave for the Helmholtz equation}.
\newblock \emph{SIAM J. Appl. Math.}, \textbf{56}, 1664--1693.
\newblock (\doi{10.1137/S0036139995281822})

\bibitem[{Blasten \emph{et~al.}(preprint)Blasten, P\"aiv\"arinta \&
  Sylvester}]{sp}
Blasten, E., P\"aiv\"arinta, L. \& Sylvester, J. preprint Do corner always
  scatter?

\bibitem[{Cakoni \emph{et~al.}(preprint)Cakoni, Chaulet \& Haddar}]{CCH13}
Cakoni, F., Chaulet, N. \& Haddar, H. preprint On the asymptotic of a {R}obin
  eigenvalue problem.
\newblock \emph{C. R. Acad. Sci. Paris}.

\bibitem[{Cakoni \emph{et~al.}(2010{\natexlab{\emph{a}}})Cakoni, Colton \&
  Haddar}]{CaCoHa10}
Cakoni, F., Colton, D. \& Haddar, H. 2010{\natexlab{\emph{a}}} {On the
  determination of Dirichlet or transmission eigenvalues from far-field data}.
\newblock \emph{C. R. Acad. Sci. Paris}, \textbf{348}, 379--383.

\bibitem[{Cakoni \emph{et~al.}(2012)Cakoni, Cossoni{\`e}re \&
  Haddar}]{CaCoHa12}
Cakoni, F., Cossoni{\`e}re, A. \& Haddar, H. 2012 {Transmission eigenvalues for
  an inhomogeneous media containing obstacles}.
\newblock \emph{Inverse Problems and Imaging}, \textbf{6}(3), 373--398.

\bibitem[{Cakoni \emph{et~al.}(2010{\natexlab{\emph{b}}})Cakoni, Gintides \&
  Haddar}]{cakginhad}
Cakoni, F., Gintides, D. \& Haddar, H. 2010{\natexlab{\emph{b}}} The existence
  of an infinite discrete set of transmission eigenvalues.
\newblock \emph{SIAM J. Math. Anal.}, \textbf{42}, 237--255.

\bibitem[{Cakoni \& Haddar(2012)}]{CakHad2}
Cakoni, F. \& Haddar, H. 2012 Transmission eigenvalues in inverse scattering
  theory.
\newblock \emph{Inverse Problems and Applications, Inside Out, 60, MSRI
  Publications}, \textbf{60}, 527--578.

\bibitem[{Giorgi \& Haddar(2012)}]{haddex}
Giorgi, G. \& Haddar, H. 2012 Computing estimates of material properties from
  transmission eigenvalues.
\newblock \emph{Inverse Problems}, \textbf{28}(5), 055\,009.

\bibitem[{Henrot(2006)}]{Hen06}
Henrot, A. 2006 \emph{{Extremum problems for eigenvalues of elliptic
  operators}}.
\newblock Birkh{\"a}user.

\bibitem[{Kirsch \& Lechleiter(to appear)}]{Leh}
Kirsch, A. \& Lechleiter, A. to appear The inside-outside duality for
  scattering problems by inhomogeneous media.
\newblock \emph{Inverse Problems}.

\bibitem[{Lakshtanov \& Vainberg(preprint)}]{LakVain13}
Lakshtanov, E. \& Vainberg, B. preprint Applications of elliptic operator
  theory to the isotropic interior transmission eigenvalue problem.
\newblock \emph{Inverse Problems}.

\bibitem[{McLean(2000)}]{McL00}
McLean, W. 2000 \emph{{Strongly elliptic systems and boundary integral
  equations}}.
\newblock Cambridge University Press.

\bibitem[{Oleinik \emph{et~al.}(1992)Oleinik, Shamaev \& Yosifian}]{OlShYo92}
Oleinik, O.~A., Shamaev, A.~S. \& Yosifian, G.~A. 1992 \emph{{Mathematical
  problems in elasticity and homogenization}}.
\newblock Elsevier.

\bibitem[{P\"aiv\"arinta \& Sylvester(2008)}]{paisyl}
P\"aiv\"arinta, L. \& Sylvester, J. 2008 Transmission eigenvalues.
\newblock \emph{SIAM J. Math. Anal}, \textbf{40}(2), 738--753.

\end{thebibliography}


\appendix
\section{Auxiliary  Regularity Estimates}\label{app1}
We start by establishing some crucial elliptic regularity estimates with explicit dependence of the constants on $\delta$. In the following, we denote by $C$ a 
generic constant independent of $\delta$. In the next Lemma we adopt the notations and the definitions of \cite[Section 4]{McL00}, which we recall here in a 
simplified setting. Let $\mathcal{O}$ be a connected Lipschitz domain of $\bR^2$ and denote by $(x_1,x_2)$ the coordinates of a point $x$ in some given basis. We 
define the operator $\mathcal{P}$ in $\mathcal{O}$ by 
\[
\mathcal{P} := -\sum_{i,j} \D{}{x_i}a_{ij}\D{}{x_j} 
\]
where for $(i,j) \in \{1,2\}^2$ $a_{ij} \in C^1(\mathcal{O})$. We say that $\mathcal{P}$ is coercive if there exists a constant $C>0$ such that for all $\xi \in \bR^2$ 
and $x\in \mathcal{O}$
\[
 \sum_{i,j}a_{ij}(x) \xi_i\xi_j \geq C|\xi|^2.
\]
We call the conormal derivative the operator $\mathcal{B}_\nu$ given by 
\[
\mathcal{B}_\nu := \sum_{i}\nu_i\gamma_{\partial\mathcal{O}} \left( \sum_j a_{ij} \D{}{x_j}\right)
\]
where $\gamma_{\partial\mathcal{O}}$ is the trace operator on $\partial\mathcal{O}$ and $\nu_i$ for $i=1,2$ is the i$^\text{th}$ component of the inward normal 
vector to $\partial \mathcal{O}$.
\begin{lemma}
\label{Lestrip}
Consider $\delta>0$, $g\in H^{1/2}(\bR)$ and $f\in L^2( \bR\times(0,\delta))$. Let  $\mathcal{P}$ be  a coercive operator with coercivity constant independent of $
\delta$ and $\mathcal{B}_\nu$  the associated conormal derivative. If $w \in H_0 := \{w \in H^1(\bR\times(0,\delta)) \,,\; w(x_1,\delta)=0 \; \mbox{for all}\, x_1\in \bR\}
$ solves
 \[\begin{cases}
\mathcal{P}w = f \text{ in } \bR\times(0,\delta) ,\\
\mathcal{B}_\nu w(x_1,0) = g(x_1) \text{ for all }  x_1 \in \bR, \\
w(x_1,\delta) = 0 \text{ for all } x_1 \in \bR , \\
\end{cases}
\]
 then there exists a constant $C>0$ independent of $\delta$ such that
\[
\|w\|_{H^{2}(\bR\times(0,\delta) )} \leq C\left(  \|f\|_{L^{2}(\bR\times(0,\delta) )}+ \|g\|_{H^{1/2}(\bR)}\right).
\]
\end{lemma}
\begin{remark}
The novel important aspect  in the above a priori estimate is to show that the constant is independent of $\delta$, and to our knowledge such a result was not 
available in the literature.
\end{remark}
\begin{proof}
Using Green's formula we have that 
\[
\int_{\bR\times(0,\delta)}\mathcal{P} w \psi \,dx_1dx_2 =\Phi( w, \psi ) + \int_\bR \mathcal{B}_\nu w(x_1,0) \psi(x_1,0) \, dx_1
\]
or
\begin{equation}
\label{eq:vfchiv}
\int_{\bR\times(0,\delta)}f\psi \,dx_1dx_2 = \Phi(w, \psi ) +\int_\bR g(x_1)\psi(x_1,0) \, dx_1
\end{equation}
 for all $\psi \in H_0$, where the bilinear form $\Phi(\cdot,\cdot)$ is defined  by
\[
\Phi(u, v ) =\sum_{i,j} \int_{\bR\times(0,\delta)} a_{ij} \D{u}{x_i}\D{v}{x_j} \, dx_1dx_2\qquad \mbox{for all  $(u,v) \in H_0$}.
\]
In order to obtain the desired regularity result we apply  the approach of the difference quotient in the direction $x_1$. To this end  for $h\in \bR$ and all $
(x_1,x_2)\in \bR\times[0,\delta]$  we define the difference quotient by
\[
\Delta_h u := \cfrac{u(x_1+h,x_2) - u(x_1,x_2)}{h}.
\]
Straightforward algebraic calculations show that the following formulas hold true
\begin{equation}
\label{eq:quot1}
\int_{\bR\times(0,\delta)} (\Delta_h u) v\,dx_1dx_2 = -\int_{\bR\times(0,\delta)} u  (\Delta_{-h} v)\,dx_1dx_2,
\end{equation}
\begin{equation}
\label{eq:quot2}
\left|\Phi(\Delta_h u, v ) + \Phi( u, \Delta_{-h}v )\right| \leq C\|v\|_{H^1(\bR\times(0,\delta))} \|u\|_{H^1(\bR\times(0,\delta))},
\end{equation}
 for all  $u$ and $v$ in $H_0$, and moreover there exists a constant $C>0$ independent of $\delta$ such that for all $u\in H_0$ and $h$ sufficiently small
\begin{equation}
\label{eq:equiv}
 C\left\| \D{u}{x_1}\right\|_{L^2(\bR\times(0,\delta))}\leq\|\Delta_h u\|_{L^2(\bR\times(0,\delta))} \leq \left\| \D{u}{x_1}\right\|_{L^2(\bR\times(0,\delta))}
\end{equation}
(see \cite{McL00}, Lemma 4.13) for the proof of this last result). Substituting (\ref{eq:quot1}) and (\ref{eq:quot2}) in  (\ref{eq:vfchiv}) we  obtain that for $\psi = 
\Delta_{-h}\Delta_h  w \in H_0$ and all $h$ sufficiently small
\begin{align}
|\Phi(\Delta_h  w&, \Delta_h  w)| \leq C \|w\|_{H^1(\bR\times(0,\delta))} \|\Delta_h  w\|_{H^1(\bR\times(0,\delta))}  \notag \\
&+\left| \int_\bR g(x_1) (\Delta_{-h}\Delta_h  w) (x_1,0)\, dx_1 + \int_{\bR\times(0,\delta)}f (\Delta_{-h}\Delta_h  w) \,dx_1dx_2\right| . \label{eq:vfdelta}
 \end{align}
From \cite{McL00}, Exercise 4.4, we have that or all $s \in \bR$
\[
\|\Delta_h u\|_{H^s(\bR)} \leq \left\|\D{u}{x_1}  \right\|_{H^s(\bR)} 
\]
provided  it is known that $\displaystyle{\D{u}{x_1}} \in H^s(\bR)$. On the other hand for the boundary term in (\ref{eq:vfdelta}) we have 
\begin{align}
&  \int_\bR |g \Delta_{-h}\Delta_h  w |\, dx_1 \leq \| g\|_{H^{1/2}(\bR)}  \|\Delta_{-h}\Delta_h  w\|_{H^{-1/2}(\bR)} \notag \\
& \qquad \qquad \leq  \| g\|_{H^{1/2}(\bR)} \|\Delta_h  w\|_{H^{1/2}(\bR)} \leq C \| g\|_{H^{1/2}(\bR)} \|\Delta_h  w\|_{H^{1}(\bR\times(0,\delta))} \label{eq:bd1}
\end{align}
with a constant $C>0$ independent of $\delta$ and $h$ (see the trace Lemma \ref{Le:Trace} for the last inequality). Finally, the coercivity of $\Phi$ together with  
(\ref{eq:equiv}), (\ref{eq:vfdelta}) and (\ref{eq:bd1}) give
\[
\|\Delta_h  w\|_{H^1(\bR\times(0,\delta))} \leq C \left(\| g\|_{H^{1/2}(\bR)}+ \|w\|_{H^1(\bR\times(0,\delta))} + \|f\|_{L^2(\bR\times(0,\delta))}\right)
\] 
which in view of \cite{McL00}, Lemma 4.13, gives
\[
\left\|\D { w}{x_1}\right\|_{H^1(\bR\times(0,\delta))} \leq C \left(\| g\|_{H^{1/2}(\bR)}+ \|w\|_{H^1(\bR\times(0,\delta))} + \|f\|_{L^2(\bR\times(0,\delta))}\right).
\]
To estimate the second order derivative with respect to $x_2$, we recall that $\mathcal{P} w = f$ and since $\mathcal{P}$ is coercive there exists a constant 
$C>0$ that depends on the coefficients $a_{ij}$ but not on $\delta$ such that
\begin{align*}
\left\|\D {^2w}{x_2^2}\right\|_{L^2(\bR\times(0,\delta))} &\leq C\left( \|f\|_{L^2(\bR\times(0,\delta))} +  \left\|\D {w}{x_1}\right\|_{H^1(\bR\times(0,\delta))}\right) \\
&\leq C \left(\| g\|_{H^{1/2}(\bR)}+ \|w\|_{H^1(\bR\times(0,\delta))} + \|f\|_{L^2(\bR\times(0,\delta))}\right).
\end{align*}
Then the result is a consequence of  the standard a priori estimate for  $ \|w\|_{H^1(\bR\times(0,\delta))}$ making use of the coercivity of $\mathcal{P}$.
\end{proof}

From the above regularity result  in a strip it is now possible to obtain the same type of regularity result in $\lay$ first with homogeneous  mixed boundary 
conditions which is stated in Lemma \ref{lema1} and then with inhomogeneous mixed boundary conditions which is stated in Lemma \ref{lema2}. 

\begin{lemma}\label{lema1}
\label{LeRegInd}
There exists  a constant $C>0$ independent of $\delta$ such that  for all $f\in L^{2}(\lay)$ the unique solution $w \in H^1(\lay)$ of 
\[\begin{cases}
-\Delta w = f \text{ in }  \lay ,\\
\displaystyle \D{w}{\ni} = 0 \text{ on }  \Ga , \\
w = 0 \text{ on }  \Ga_{\delta} , \\
\end{cases}
\]
satisfies the a priori estimate 
\[
\|w\|_{H^{2}(\lay)} \leq C \|f\|_{L^{2}(\lay)}.
\]

\end{lemma}
\begin{proof}
First from  the standard a priori estimate for the Laplacian we have that 
\begin{equation}
\label{eq:EstH1}
\|w\|_{H^{1}(\lay)} \leq C\|f\|_{L^{2}(\lay)}
\end{equation}
with a constant $C>0$ independent of $\delta$.
\begin{figure}[hh]
\begin{center}
\includegraphics[width=.35\textwidth]{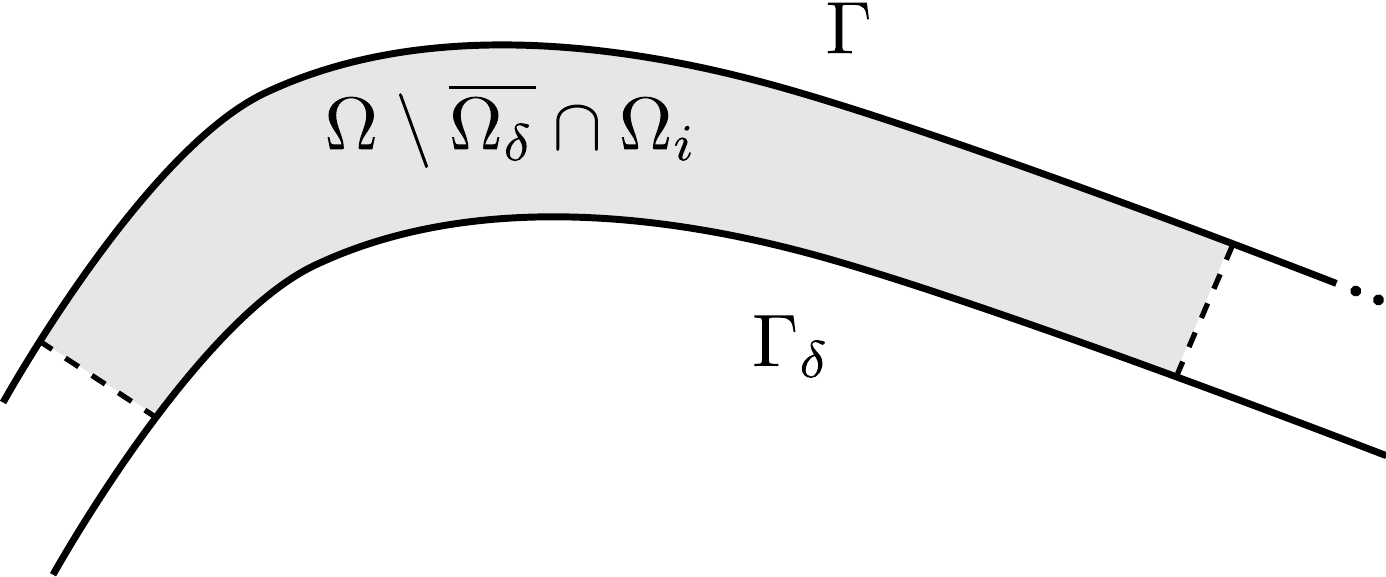}
\end{center}
\caption{Local covering of the layer}
\label{fig:cover}
\end{figure}
To obtain $H^2$ estimates our approach is based on first locally straighten the boundary and then apply Lemma \ref{Lestrip}.  To this end,  since $\b{\lay}$ is a 
compact set, there exists an integer $n$ and a sequence $(\Om_i)_{i=1,\cdots,n}$ of bounded and connected domains of $\bR^2$ such that $\b{\lay} \subset 
\cup_{i=1}^{n} \Om_i$  for all $\delta$ sufficiently small. Moreover, we take  $\Om_i$ such that there exists $s_i>0$ and a $C^{\infty}([-s_i,s_i])$, $k\geq 0$, function 
$x_\Ga$ such that for all $\delta$ sufficiently small we have 
\[
(\lay) \cap \Om_i =\{x_\Ga(s) + \eta \ni(s) \,,\forall (s,\eta)\in (-s_i,s_i)\times (0,\delta) \}
\]
where $x_\Ga(s) \in \Ga$ for $s\in (-s_i,s_i)$ and $i=1,\cdots,n$ (see Figure \ref{fig:cover}). Thus for all $\delta$ sufficiently small, 
\begin{align*}
\vp_i \,:\quad \b{ (\lay) \cap \Om_i }&\longrightarrow \b{\hat{\Om}_i}\\
 x&\longmapsto (s,\eta)
\end{align*}
is a $C^{2}$-diffeormorphism, where $\hat{\Om}_i:=(-s_i,s_i)\times (0,\delta)$.  Let $(\phi_i)_{i=1,\cdots, n}$ be a partition of unity such that  $\phi_i\in C^
\infty(\bR^2)$, $\text{supp} (\phi_i) \subset \Om_i$ and $\sum_i \phi_i =1$ in $\b{\lay}$. Hence if we define $w_i:= \phi_i w \in H^1(\lay)$, then 
\begin{equation}
\label{eq:unity}
\|w\|_{H^{2}(\lay)}  \leq \sum_{i=1}^N \|w_i\|_{H^{2}(\lay)}
\end{equation}
and $w_i$ is compactly supported in $\Om_i$. Furthermore, $w_i$ solves
\[\begin{cases}
-\Delta w_i = f_i \text{ in }  \lay ,\\
\displaystyle \D{w_i}{\ni} = g_i \text{ on }  \Ga , \\
w_i = 0 \text{ on }  \Ga_{\delta} , \\
\end{cases}
\]
with
\[
f_i := f\phi_i+ 2 \nabla w\nabla\phi_i + w\Delta\phi_i \qquad \mbox{and}\qquad g_i:=w\D{\phi_i}{\ni}.
\]
In the following, for $U\in L^2((\lay) \cap \Om_i )$ and $G\in H^{s}(\Ga\cap\partial \Om_i)$ we denote by $\wt{U} \in L^2(\hat{\Om}_i)$ and $\tilde G\in H^s((-
s_i,s_i))$  the functions defined by $\wt{U}  := U \circ \vp_i$ and $\wt{G}  := G \circ \vp_i$,  respectively. Since the $W^{2,\infty}$ norm of $\vp_i$ and $\vp_i^{-1}$  
does not depend on $\delta$ it is easy to see that there exists a constant $C>0$ independent of $\delta$ such that for $U\in H^p((\lay) \cap \Om_i)$
\begin{equation}\label{eq:eqnorm1}
\cfrac{1}{C} \|\wt{U}\|_{H^p(\hat{\Om}_i)} \leq \|U\|_{H^p((\lay) \cap \Om_i)} \leq C \|\wt{U}\|_{H^p(\hat{\Om}_i)} \qquad \mbox{for all integer $0<p\leq2 $}
\end{equation}
and for $G\in H^{s}(\Ga\cap\partial \Om_i)$
\begin{equation}\label{eq:eqnorm2}
\cfrac{1}{C} \|\wt{G}\|_{H^s((-s_i,s_i))} \leq \|G\|_{H^s(\Ga\cap\partial \Om_i)} \leq C \|\wt{G}\|_{H^s((-s_i,s_i)}\qquad \mbox{for all $0<s\leq2$}.
\end{equation}
With these notations, we can now prove using the calculations developed in Section \ref{sec:devas} that $\wt{w}_i \in H^1(\hat{\Om}_i)$ solves
\begin{equation}
\label{eq:locvar}\begin{cases}
\mathcal{P} \wt{w}_i = \wt{f}_i \text{ in }  \hat{\Om}_i,\\
(\mathcal B_\nu \wt{w}_i)(s,0) = \wt{g}_i(s) \text{ for all }   s \in (-s_i,s_i), \\
\wt{w}_i(s,\delta) = 0 \text{ for all }  s \in (-s_i,s_i), \\
\end{cases}
\end{equation}
where $\wt{f}_i(s,\eta) := (1+\eta \kappa(s))(f_i \circ \vp_i)(s,\eta)$, $\wt{g}_i(s,\eta):=-(1+\eta \kappa(s))(g_i\circ\vp_i)(s,\eta)$ and $\mathcal{P}$ and $\mathcal{B}_
\nu$ are as in Lemma \ref{Lestrip} (note that the coercivity constant for $\mathcal{P}$ does not depend on $\delta$). Furthermore since $\wt{w}_i$, $\wt{f}_i$ and $
\wt{g}_i$ are equal to $0$ in a vicinity of $-s_i$ and $s_i$ we can extend them by $0$ into $\bR\times (0,\delta)$. For sake of simplicity, we do not change the 
notations for their extension and note that these extension also satisfy the system (\ref{eq:locvar}) for  $s\in \bR$. Hence we can apply Lemma \ref{Lestrip} to $
\wt{w}_i$ to obtain
\[
\|\wt{w}_i\|_{H^2(\bR\times(0,\delta))}\leq C\left(\|\wt{g}_i\|_{H^{1/2}(\bR)}+ \|\wt{f}_i\|_{L^2(\bR\times(-s_i,s_i))}\right)
\]
where $C$ is independent of $\delta$. Using (\ref{eq:eqnorm1}) and  (\ref{eq:eqnorm2}) in  (\ref{eq:unity}) together with our first a priori estimate (\ref{eq:EstH1}) 
finally proves the lemma.
\end{proof}

Next we obtain the same type of estimates as in Lemma \ref{LeRegInd} for inhomogeneous  boundary condition on $\Ga$. The challenge  is to show that the 
lifting function is bounded independently of $\delta$ in appropriate norm.
\begin{lemma}\label{lema2}
\label{LeBoundUnif} 
Let $g\in H^{1/2}(\Ga)$ and $w \in H^1(\lay)$ be the unique solution of 
\[\begin{cases}
\Delta w = 0 \text{ in } \lay ,\\
\displaystyle \D{w}{\ni} = g \text{ on }  \Ga , \\
w = 0 \text{ on } \Ga_{\delta}. \\
\end{cases}
\]
Then there exists a constant $C>0$ such that for all $\delta>0$ 
\[
\|w\|_{H^{2}(\lay)} \leq C \|g\|_{H^{1/2}(\Ga)}.
\]
In addition, if $\|g\|_{H^{1/2}(\Ga)}\leq C_g$ for some $C_g>0$  and all $\delta >0$, then
\[
\|w\|_{H^{2}(\lay)} \underset{\delta \to 0}{\longrightarrow} 0.
\]
\end{lemma}
\begin{proof}
We start by building an appropriate lifting of $g$ which equals to $0$ on $\Ga_{\delta}$. To this end, let us define $g_i=\phi_i g$ and its local counterpart $\wt{g}_i
$ using the same partition of unity and local parameterization as in Lemma \ref{LeRegInd}. Then we can define an extension of $\wt{g}_i$ to $\bR$ denoted by $
\b{g}_i$ by 
\[
\b{g}_i(s) := 
\begin{cases}
\wt{g}_i(s) \text{ if } s\in[-s_i,s_i] ,\\
0 \text{ elsewhere.}
\end{cases}
\]
For any function $G \in L^1(\bR)$ let 
\[
\mc{F} G(\xi) := \int_{\bR}G(s) e^{-i2\pi s\xi} \,ds \qquad \mbox{and}\qquad  
\mc{F}^{-1} G(s) := \int_{\bR}G(\xi) e^{i2\pi s\xi} \,d\xi
\]
be its Fourier transform and its inverse Fourier transform, respectively. Since $\Ga$ is of class $C^{2}$ and  $g\in H^{1/2}(\Ga)$, Plancherel's Theorem ensure the 
existence of a constant $C$ independent of $\delta$ such that
\begin{equation}
\label{eq:fourier}
\|(1+\xi^2)^{1/4} \mc{F} \b{g}_i\|_{L^2(\bR)} \leq C \|g\|_{H^{1/2}(\Ga)}.
\end{equation}
For all $\xi \in \bR$ and $\eta \in [0,\delta]$ let us define 
\[
w_i(\xi,\eta) := \cfrac{\mc{F} \b{g}_i(\xi)}{|\xi|}\left(\cfrac{\sinh(|\xi|(\eta-\delta))}{\cosh(|\xi|\delta)} \right)
\]
and
\[
\b{w}_{i}(s,\eta) := (\mc{F}^{-1}w_i)(s,\eta).
\]
Then $\b{w}_{i}$ satisfies $\b{w}_i(s,\delta)=0$ and $\D{\b{w}_i}{\eta}(s,0) = \b{g}_i(s)$ in $\bR$.  Moreover,  there exists a constant  $C>0$ independent of $\delta
$ such that for all $(\xi,\eta)\in \bR\times(0,\delta)$ we have
\[
|w_{i}(\xi,\eta) |^2 \leq  C (\mc{F} \b{g}_i(\xi))^2\,,\quad |\xi w_{i}(\xi,\eta) |^2 \leq  C (\mc{F} \b{g}_i(\xi))^2\,,\quad\left|\D{w_{i}}{\eta}(\xi,\eta) \right|^2 \leq  C (\mc{F} 
\b{g}_i(\xi))^2
\]
Integrating the above inequalities over $\bR\times(0,\delta)$ and  using (\ref{eq:fourier}) and the Plancherel's Theorem we have that there exists a constant $C>0$ 
independent of $\delta$ such that
\begin{equation*}
\|\b{w}_i \|^{2}_{H^{1}(\bR\times (0,\delta))} \leq  C\delta \|g\|^{2}_{H^{1/2}(\Ga)}.
\end{equation*}
Moreover, for all $(\xi,\eta)\in \bR\times(0,\delta)$ we also have
\begin{align*}
 \xi^2 w_{i}(\xi,\eta)&=  |\xi|\mc{F} \b{g}_i(\xi)\frac{\sinh(|\xi|(\eta-\delta))}{\cosh(|\xi|\delta)},\\
 \xi\D{w_{i}}{\eta}(\xi,\eta)&=  |\xi|\mc{F} \b{g}_i(\xi)\frac{\cosh(|\xi|(\eta-\delta))}{\cosh(|\xi|\delta)},\\
 \D{^2w_{i}}{\eta^2}(\xi,\eta)&=  |\xi|\mc{F} \b{g}_i(\xi)\frac{\sinh(|\xi|(\eta-\delta))}{\cosh(|\xi|\delta)} .
\end{align*}
Therefore 
\begin{align}
\int_{\bR} \int_{0}^\delta |\xi|^4 |w_{i}(\xi,\eta)|^2 d\eta d\xi &=  \int_{\bR}  |\xi|^2|\mc{F} \b{g}_i(\xi)|^2\int_{0}^\delta \frac{\sinh^2(|\xi|(\eta-\delta))}{\cosh^2(|\xi|\delta)}
d\eta d\xi \nonumber\\
&= \int_{\bR}    |\xi|^2|\mc{F} \b{g}_i(\xi)|^2\int_{0}^\delta  \frac{\cosh(2|\xi|(\eta-\delta))-1}{2\cosh(|\xi|\delta)^2}d\eta d\xi  \nonumber\\
&=\frac{1}{2} \int_{\bR}    |\xi|^2|\mc{F} \b{g}_i(\xi)|^2\left( \frac{\sinh(|\xi|\delta)}{|\xi|\cosh(|\xi|\delta)} - \frac{\delta}{\cosh(|\xi|\delta)^2}\right)d\xi \nonumber\\
&\leq \frac{1}{2} \int_{\bR}  |\xi||\mc{F} \b{g}_i(\xi)|^2 \frac{\sinh(|\xi|\delta)}{\cosh(|\xi|\delta)} d\xi \label{eq:lebesgue}\\
& \leq  \frac{1}{2} \int_{\bR}  |\xi||\mc{F} \b{g}_i(\xi)|^2  d\xi, \nonumber
\end{align}
and hence  an application of  \eqref{eq:fourier} yields
\[
\int_{\bR} \int_{0}^\delta |\xi|^4 |w_{i}(\xi,\eta)|^2 d\eta d\xi \leq C \|g\|^{2}_{H^{1/2}(\Ga)}.
\]
In a similar way we obtain that
\[
\int_{\bR} \int_{0}^\delta \left|\xi\D{w_{i}}{\eta}(\xi,\eta)\right| d\eta d\xi \leq C \|g\|^{2}_{H^{1/2}(\Ga)}\quad \text{ and } \quad \int_{\bR} \int_{0}^\delta \left|\D{^2w_{i}}
{\eta^2}(\xi,\eta)\right| d\eta d\xi \leq C \|g\|^{2}_{H^{1/2}(\Ga)}
\]
with a different constant $C>0$ independent of $\delta$. Hence,  once more application of Plancherel's Theorem implies
\begin{equation}
\label{eq:borneloc}
\|\b{w}_i \|^{2}_{H^{2}(\bR\times (0,\delta))} \leq C \|g\|^{2}_{H^{1/2}(\Ga)}.
\end{equation}

Next assume  in addition that there exists $C_g>0$ such that  $\|g\|_{H^{1/2}(\Ga)}\leq C_g$ for all $\delta >0$, then for almost every $\xi \in \bR$,
\[
 |\xi||\mc{F} \b{g}_i(\xi)|^2 \frac{\sinh(|\xi|\delta)}{\cosh(|\xi|\delta)}  \underset{\delta \to 0}{\longrightarrow} 0
\]
and  since
\[
|\xi|\mc{F} \b{g}_i(\xi)|^2 \frac{\sinh(|\xi|\delta)}{\cosh(|\xi|\delta)}  \leq  |\xi|(\mc{F} \b{g}_i(\xi))^2,
\]
the Lebesgue dominated convergence theorem with \eqref{eq:lebesgue} implies that
\[
\int_{\bR} \int_{0}^\delta |\xi|^4 |w_{i}(\xi,\eta)|^2 d\eta d\xi \underset{\delta \to 0}{\longrightarrow} 0
\]
In a similar way we obtain that
\[
\int_{\bR} \int_{0}^\delta \left|\xi\D{w_{i}}{\eta}(\xi,\eta)\right| d\eta d\xi  \underset{\delta \to 0}{\longrightarrow} 0 \quad \text{ and } \quad \int_{\bR} \int_{0}^\delta \left|
\D{^2w_{i}}{\eta^2}(\xi,\eta)\right| d\eta d\xi \underset{\delta \to 0}{\longrightarrow} 0,
\]
whence, again from the Plancherel's Theorem,
\begin{equation}
\label{eq:relloc}
\|\b{w}_i \|^{2}_{H^{2}(\bR\times (0,\delta))}\underset{\delta \to 0}{\longrightarrow} 0.
\end{equation}
 Now we go back to the physical domain and define for $x \in \lay$
\[
w_{g}(x) :=\sum_{j\text{ s.t. }x \in \Om_j}  (\b{w}_{j} \circ \vp_j^{-1})(x).
\]
Then $w_{g}$ satisfies $\D{w_{g}}{\ni}|_{\Ga}=g$ together with $w_{g}|_{\Ga_{\delta}} = 0$ and using the fact that $\Ga$ is of class $C^{2}$ and that  $(\b{w}_i)_i$ 
satisfy (\ref{eq:borneloc}), we can claim that there exists a  constant $C>0$ independent of $\delta$ such that
\begin{equation}
\label{eq:boundg}
\|w_{g} \|^{2}_{H^{2}(\lay)} \leq  C \|g\|^{2}_{H^{1/2}(\Ga)}.
\end{equation}
In addition, if there exists $C_g>0$ such that  $\|g\|_{H^{1/2}(\Ga)}\leq C_g$ for all $\delta >0$, then from (\ref{eq:relloc}) we also have that 
\begin{equation}
\label{eq:convg}
\|w_{g} \|^{2}_{H^{2}(\lay)}\underset{\delta \to 0}{\longrightarrow} 0.
\end{equation}
Finally we consider $W := w-w_{g}$ which solves 
\[\begin{cases}
\Delta W = - \Delta w_{g} \,\text{in} \, \lay ,\\
\D{W}{\ni} = 0 \,\text{on} \, \Ga , \\
W = 0 \,\text{on} \, \Ga_{\delta} , \\
\end{cases}
\]
 Lemma \ref{LeRegInd} states that there exists a constant $C$ independent of $\delta$ such that
\[
\|W\|_{H^{2}(\lay)} \leq C \|w_{g} \|^{2}_{H^{2}(\lay)}.
\]
This last estimate together with (\ref{eq:boundg}) imply the existence of a constant $C$ independent of $\delta$ such that
\[
\|w\|_{H^{2}(\lay)} \leq C \|g\|_{H^{1/2}(\Ga)}.
\] 
and, if in addition there exists $C_g>0$ such that $\|g\|_{H^{1/2}(\Ga)}\leq C_g$  for all $\delta >0$,  (\ref{eq:convg}) implies 
 \[
\|w\|_{H^{2}(\lay)}\underset{\delta \to 0}{\longrightarrow} 0,
\] 
which ends the proof.
\end{proof}

Next we prove  a trace theorem which displays explicit dependence on $\delta$ of the constant.
\begin{lemma}
\label{Le:Trace}
For $k=1,2$, there exists a constant $C>0$ independent of $\delta$ such that for all $v\in H^k(\lay)$ with $v|_{\Ga_\delta} = 0$ 
\[
\|v\|_{H^{k-1/2}(\Ga)} \leq C \|v\|_{H^k(\lay).}
\]
\end{lemma}
\begin{proof}
We prove the result for the domain $\bR\times(0,\delta)$ and then use the partition of unity and  change of variable introduced in the proof of Lemma 
\ref{LeRegInd}  to obtain the desired result. To this end, we  consider an arbitrary $v \in C^\infty(\bR^2)$ with compact support in $\bR^2$ and denote as before by 
$\mc{F} v$ its Fourier transform  with respect to the the first variable where $\xi$ denotes the dual variable. Integrating by parts, we obtain that
 \begin{align*}
 (1+\xi^2)^{2(k-1/2)}|\mc{F} v(\xi,0)|^2 &= (1+\xi^2)^{2(k-1/2)}|\mc{F} v(\xi,\delta)|^2 \\
 &- 2 \Re\left( \int_0^\delta (1+\xi^2)^{k-1} \D{\mc{F} v}{\eta}(\xi,\eta) (1+\xi^2)^{k} \mc{F} v(\xi,\eta)\,d\eta \right) 
 \end{align*}
 for all $\xi \in \bR$. Integrating this equality along $\bR$, and using the Cauchy-Schwartz inequality and Plancherel's Theorem imply the existence of $C>0$ 
independent of $\delta$ such that
 \[
 \|v(s,0)\|_{H^{k-1/2}(\bR)} \leq C \left(\|v\|_{H^k(\bR\times(0,\delta))} +  \|v(s,\delta)\|_{H^{k-1/2}(\bR)}\right).
 \]
 A density argument ensures that the above estimate holds for all $v\in H^2( \bR\times(0,\delta))$, and thus we obtain the result  for $v$ whose trace is $0$ on $
\bR\times\{\delta\}$.
\end{proof}

We conclude this section with the following  technical trace lemma, which was already used in the proof of Lemma \ref{Lestrip}

\begin{lemma}
\label{LeTrace}
For any $w \in H^{1}(\lay)$ we have that 
\begin{equation}
\label{eq:TraceVol}
\|w\|_{L^{2}(\lay)} \leq C(\delta^{1/2}\|w\|_{L^{2}(\Ga)} + \delta \|w\|_{H^{1}(\lay)}),
\end{equation}
and  for any function $w \in H^{2}(\lay)$ such that $w|_{\Ga_{\delta}}=0$  we have that 
\begin{equation}
\label{eq:TraceBord}
\|w\|_{H^{1/2}(\Ga)} \leq C \delta^{1/2} \|w\|_{H^{2}(\lay)},
\end{equation}
where the constant $C>0$ is independent of $\delta$. 
\end{lemma}
\begin{proof}
To prove (\ref{LeTrace}) we take $\phi \in C_{0}^{\infty}(\bR^{2}) $  and using local variables in the layer to obtain  
\[
\wt{\phi}(s,\eta) = \wt{\phi}(s,0) + \int_{0}^{\eta} \D{\wt{\phi}}{t}(s,t) \,  dt ds \qquad \mbox{for every $\eta \leq \delta$},
\]
where $\tilde \phi$ denotes the function in the new variables $(s,\eta)$ (see Lemma \ref{LeRegInd} for the definition). The Cauchy-Schwartz inequality implies 
\[
|\wt{\phi}(s,\eta)|^{2} \leq C\left(|\wt{\phi}(s,0)|^{2} + \delta \|\phi\|^{2}_{H^{1}(\lay)} \right),
\]
whence we obtain the desired estimate for $\phi \in C_{0}^{\infty}(\bR^{2}) $ by integrating this inequality over $\Ga \times [0,\delta]$. A density argument gives the 
result for all $w \in H^{1}(\lay)$. \\
To prove  (\ref{eq:TraceBord}) we again take $\phi \in C_{0}^{\infty}(\bR^{2})$ and in a similar way as above use local  coordinates in the layer. Thus we obtain 
that for $s \in [0,s_0]$
\[
|\wt{\phi}(s,0)|\leq |\wt{\phi}(s,\delta)| + \left| \int_{0}^{\delta} \D{\wt{\phi}}{\eta}(s,\eta)  \,d\eta \right| \leq  |\wt{\phi}(s,\delta)| + \delta^{1/2} \|\phi\|_{H^{1}(\lay)}
\]
which leads to 
\[
\|\phi\|_{L^{2}(\Gamma)}\leq C\left( \|\phi\|_{L^{2}(\Gamma_{\delta})}+\delta^{1/2} \|\phi\|_{H^{1}(\lay)}\right)
\]
for some positive $C>0$ independent of $\delta$. Similarly we also have
\[
\|\nabla_{\Ga}\phi\|_{L^{2}(\Gamma)}\leq C \left( \|\nabla_{\Ga}\phi\|_{L^{2}(\Gamma_{\delta})}+\delta^{1/2} \|\phi\|_{H^{2}(\lay)}\right).
\]
By density, the above  two inequalities remain true for all $w\in H^{2} (\lay)$ and by interpolation between $L^2(\Ga)$ and $H^1(\Ga)$ we obtain the desired 
result notting that  $w|_{\Ga_{\delta}}=0$.
\end{proof}
\section{Perturbation of an Eigenvalue Problem}
We recall here some  known results about  the convergence of eigenvalues for self adjoint, positive and compact operators. The proof of the following 
fundamental result can be found in Section 3 of \cite{OlShYo92}.
\begin{theorem}
\label{th:ConvEigen}
Assume that $A \,:\, H \rightarrow H$ is a linear self-adjoint positive and compact operator on an Hilbert space $H$. Let $u \in H$ be such that $\|u\|_H =1$ and $
\ld$, $r>0$ such that 
\[
\|Au- \ld u\|_H\leq r.
\]
Then there is an eigenvalue $\ld_i$ of the operator $A$ satisfying 
\[
|\ld-\ld_i|\leq r.
\]
Furthermore, for any $r^*>r$ there exists $u^*\in H$ with $\|u^*\|_H=1$ belonging to the eigenspace associated with all the eigenvalues of the operator $A$ lying 
in $[\ld-r^*,\ld+r^*]$ that  satisfies 
\[
\|u-u^*\|_H \leq \frac{2r}{r^*}.
\]
\end{theorem}
Based on this general result, we can obtain the following lemma for the Laplace operator with Dirichlet boundary conditions which is used in our asymptotical 
analysis in the main body of the paper.
\begin{lemma}
\label{Le:ConvEigen} 
Let $\ld_i$ be a simple eigenvalue  of the negative  Laplacian with Dirichlet boundary conditions in $\Om$. Assume that it exists $u \in H^1_0(\Om)$ and 
$r$ such that
\begin{equation}
\label{eq:varequ}
\left|\int_\Om \nabla u \cdot \nabla v - \ld_i uv \, dx\right| \leq  r \|v\|_{H^1_0(\Om)} \quad \forall v \in H^1_0(\Om),
\end{equation}
and
\begin{equation}
\label{eq:normur}
|\|u\|_{L^2(\Om)}-1| \leq r.
\end{equation}
 Then it exists an eigenfunction $u_i$  associated with the eigenvalue $\ld_i$ and normalized as $\|u_i\|_{L^2(\Om)}=1$  such that
\[
\|u - u_i\|_{H^1(\Om)} \leq C r.
\]
for some  constant $ C>0$ independent of $r$ and $u$.
\end{lemma}
\begin{proof}
The proof is essentially based on Theorem \ref{th:ConvEigen}. We define the operator $A \,: H^1_0(\Om) \rightarrow H^1_0(\Om)$ by
\[
(Au,v)_{H_0^1(\Om)} := (u,v)_{L^2(\Om)} \quad \forall (u,v)\in H^1_0(\Om)
\]
where for all  $(u,v)\in H^1_0(\Om)$,
\[
(u,v)_{H^1_0(\Om)} := \int_{\Om} \nabla u\cdot\nabla v\,dx
\]
and $(u,v)_{L^2(\Om)} $ is the usual $L^2$ scalar product. Obviously $A$ is a self-adjoint positive and compact operator on $H^1_0(\Om)$,  and hence it has a 
discrete spectrum $(1/\ld_i)_{i=1,\cdots,\infty}$ such that
\[
0<\ld_1\leq \cdots\leq\ld_i\leq \cdots \quad \underset{ i \longrightarrow \infty}{\longrightarrow \infty}.
\]
and their associated eigenfunctions $u_i \in H^1_0(\Om)$ satisfy
\[
-\Delta u_i = \ld_i u_i.
\]
By hypothesis, the function $\wt u:= u/\|u\|_{H^1_0(\Om)}$ satisfies
\[
\left\| A \wt u - \frac{1}{\ld_i} \wt u \right\|_{H^1_0(\Om)}\leq \frac{r}{\ld_i \|u\|_{H^1_0(\Om)}}.
\]
 Therefore, since $\ld_i$ is supposed to be simple, the second part of Theorem \ref{th:ConvEigen} ensures the existence of an eigenfunction  $\wt{u}_i \in H^1_0(\Om)$   
of $A$ associated with $1/\ld_i$ and normalized as $\|\wt{u}_i\|_{H^1_0(\Om)}=1$, such that
\begin{equation}
\label{eq:convtilde}
\|\wt{u} - \wt{u}_i\|_{H^1_0(\Om)} \leq C_1 \frac{r}{\|u\|_{H^1_0(\Om)}},
\end{equation}
for some constant $C_1>0$ that only depends on $\ld_i$ and on the distance between $1/\ld_i$ and the closest eigenvalue of $A$. To end the proof we must renormalize this 
last inequality. 

Let us introduce $u_i:= \wt u_i/\|\wt u_i\|_{L^2(\Om)}$, then  \eqref{eq:convtilde} gives
\begin{equation}
\label{eq:appconcl1}
\|u - u_i\|_{H^1_0(\Om)} \leq C_1 r + \left \|u_i - \|u\|_{H^1_0(\Om)} \wt u_i \right\|_{H^1_0(\Om)}.
\end{equation}
By using the definition of $u_i$, the second term in this expression becomes
\begin{equation}
\label{eq:appconcl2}
\left\|u_i - \|u\|_{H^1_0(\Om)} \wt u_i\right\|_{H^1_0(\Om)}\leq \left|\sqrt{\ld_i} - \|u\|_{H^1_0(\Om)}\right| \leq \left| \ld_i -\|u\|^2_{H^1_0(\Om)}\right| \frac{1}{\sqrt{\ld_i}} 
\end{equation}
but from \eqref{eq:varequ} and \eqref{eq:normur} we have
\begin{align}
\left | \ld_i-\|u\|^2_{H^1_0(\Om)} \right| & \leq  \left|\ld_i -  \ld_i\|u\|^2_{L^2(\Om)}\right| + r \|u\|_{H^1_0(\Om)} \nonumber \\
&\leq \ld_i \left| 1 - \|u\|_{L^2(\Om)}\right|  \left|1+ \|u\|_{L^2(\Om)} \right|+ r \|u\|_{H^1_0(\Om)} \nonumber \\
& \leq \ld_i r (2+r) + r \|u\|_{H^1_0(\Om)} \label{eq:appconcl3} .
\end{align}
The $H^1_0$-norm of $u$ can be controlled by using \eqref{eq:varequ} and  Poincarre's inequality:
\[
\|u\|_{H^1_0(\Om)} \leq r + \ld_i \sqrt{\ld_0} (1+r).
\] 
 This last inequality together with \eqref{eq:appconcl1}, \eqref{eq:appconcl2} and \eqref{eq:appconcl3} give the result for a constant $C$ that only depends on 
$\ld_0$ and $\ld_i$.
\end{proof}
\end{document}